\documentclass[10pt,a4paper]{amsart}

\usepackage{amsmath,amsfonts,amsthm,amssymb,amscd,appendix,verbatim}
\usepackage{xpatch}
\usepackage[utf8]{inputenc}
\usepackage{setspace,color}
\usepackage{array}
\usepackage{enumitem}
\usepackage{cite}
\usepackage{mathtools}

\usepackage[
    a4paper,
    margin=0.7in
]{geometry}

\usepackage[
    colorlinks=true,
    linkcolor=blue,
    citecolor=blue,
    urlcolor=blue
]{hyperref}

\newtheorem{theorem}{Theorem}[section]
\newtheorem{corollary}[theorem]{Corollary}

\theoremstyle{definition}

\newtheorem{remark}[theorem]{Remark}

\newcommand{\caixa}{\hglue15.1cm$\square$\vspace{5mm}}

\newcommand{\bR}{\mathbb{R}}

\newcommand{\cF}{\mathcal{F}}

\newcommand{\cS}{\mathcal{S}}

\newcommand{\cX}{\mathcal{X}}

\newcommand{\rP}{\mathbb{P}}

\DeclareMathOperator*{\esssup}{\mathrm{ess\,sup}}

\makeatletter
\xpatchcmd{\@thm}{\fontseries\mddefault\upshape}{}{}{}
\makeatother

\title[Fractional Navier--Stokes equations with damping]
{On solutions for damped fractional Navier--Stokes equations in a critical  Lei--Lin--Gevrey space}
\author[]{Bruno S. Donato$^\ddagger$}
\thanks{$^\ddagger$Department of Mathematics and Statistics, University of Helsinki,
Helsinki 00560, Finland, e-mail: bruno.santannadonatode\-moura@helsinki.fi}

\author[Fractional Navier--Stokes equations with damping]{Wilberclay G. Melo$^*$}
\thanks{$^*$ Departamento de Matemática, Universidade Federal de Sergipe,
São Cristóvão, SE 49100-000, Brazil, e-mail: wilberclay@gmail.com}

\author[]{Thyago S. R. Santos$^\dagger$}
 \thanks{$^\dagger$Departamento de Matemática, Instituto de Matemática, Estatística e Computação Científica,
Universidade Estadual de Campinas, Campinas, SP 13083-859, Brazil, e-mail: thyagosr@unicamp.br. This author is partially supported by São Paulo Research Foundation (FAPESP)  grant 2024/15587-1}

\date{}

\begin{document}

\maketitle

\begin{abstract}
\noindent We study the fractional Navier--Stokes equations with cubic damping in the specific Lei--Lin--Gevrey space $\cX^0_{a,\sigma}(\bR^3)$ (with $a\geq0$ and $\sigma\geq1$). We establish local and small-data global well-posedness, quantitative blow-up criteria, Gevrey loss-of-radius estimates, and perturbative stability in the critical endpoint case $\gamma=\frac12$ of the usual Navier-Stokes system.
\end{abstract}
\vspace{0.5cm}



\section{Introduction}

The purpose of this work is to study existence, blow-up criteria, and stability of mild solutions for the Navier--Stokes equations with fractional dissipation and nonlinear damping in the specific Lei--Lin--Gevrey space $\cX^0_{a,\sigma}(\bR^3)$ (with $a\geq0$ and $\sigma\geq1$). Our approach relies on standard tools from Fourier analysis (see \cite{HormanderALPDO1} and the references therein), combined with estimates adapted to the convolution structure of the nonlinear terms. More precisely, we consider the following system:
\begin{equation}\label{NS}	\tag{NS}
\left\{
\begin{array}{l}
u_t
\;\!+\,
(-\Delta)^{\gamma}\,u
\,+\,
u \cdot \nabla u
\,+\, \alpha|u|^2u
\,+\,
\nabla \;\!p \:\!
\;=\;
0, \quad x\in \mathbb{R}^3, t>0;\\
\mbox{div}\:u  \;=\; 0, \quad x\in \mathbb{R}^3, t>0;\\
u(x,0) \,=\, u_0(x), \quad x\in \mathbb{R}^3,
\end{array}
\right.
\end{equation}
where $u(x,t)=(u_1(x,t),u_2(x,t),u_3(x,t))\in\mathbb{R}^3$ denotes the velocity field and $p=p(x,t)$ is the scalar pressure. The parameter $\alpha\geq0$ is the Forchheimer damping coefficient, whereas $\gamma\geq\frac12$ determines the strength of the fractional dissipation. Throughout the paper, the initial datum $u_0$ is assumed to be divergence free. As usual, the fractional Laplacian is defined through the Fourier transform by
$$
\mathcal{F}[(-\Delta)^{\gamma} f](\xi)
=
|\xi|^{2\gamma}{\widehat f(\xi)},
\qquad
\xi\in\mathbb{R}^3.
$$
The lower endpoint $\gamma=\frac12$ deserves particular attention. In this case, $(-\Delta)^{1/2}$ is an operator of order one, while the transport term $u\cdot\nabla u$ also contains one spatial derivative. Thus, the linear dissipation and the derivative carried by the quadratic nonlinearity have exactly the same differential order. This balance makes the case $\gamma=\frac12$ a borderline regime and removes the derivative advantage available when $\gamma>\frac12$. This is one of the main analytical difficulties of the problem; see, for instance, \cite{Bcritico} and the references therein.

From the physical point of view, the system \eqref{NS} describes the evolution of the velocity of an incompressible fluid subject to nonlinear transport, viscous or fractional dissipation, pressure forces, and an additional nonlinear drag. The term $u\cdot\nabla u$ represents the transport of momentum by the fluid itself and is the main source of nonlinear interaction. The pressure does not have its own evolution equation; instead, it adjusts instantaneously so that the incompressibility condition $\operatorname{div}u=0$ is preserved. In this sense, $p$ acts as a Lagrange multiplier associated with the divergence-free constraint. The term
$$
\alpha |u|^2u,\qquad \alpha\geq0,
$$
is a Forchheimer-type damping term. Its direction is opposite to the motion and its magnitude increases with the velocity. Formally, when the equation is tested against $u$, this term contributes the nonnegative quantity
$$
\alpha\int_{\mathbb R^3}|u|^4\,dx
$$
to the energy dissipation. It therefore provides an additional mechanism for suppressing large amplitudes and for counteracting the concentration produced by the nonlinear transport. From the mathematical point of view, however, the presence of this cubic term also introduces an additional nonlinear interaction that must be controlled in the solution space.

An important feature of \eqref{NS} is the fractional dissipation operator $(-\Delta)^\gamma$. When $\gamma=1$, one recovers the usual Laplacian of the classical Navier--Stokes equations (see \cite{CaiJiu2008} and references therein). For $\frac12\leq\gamma<1$, the dissipation is weaker than the classical viscous dissipation, and therefore the competition between nonlinear transport and high-frequency damping becomes more delicate. The value $\gamma=\frac12$ is especially important because $(-\Delta)^{1/2}$ has order one, exactly the same differential order as the derivative contained in $u\cdot\nabla u$. Thus, at the level of derivatives, the dissipation has only the strength required to compensate the derivative appearing in the quadratic nonlinearity, with no additional smoothing available. By contrast, when $\gamma>\frac12$, the operator $(-\Delta)^\gamma$ has order strictly larger than one, and the equation lies in a subcritical dissipative regime from the point of view of the estimates used below.

This distinction is also reflected by the natural scaling of the  fractional Navier--Stokes equations. Ignoring, for the moment, the damping term, if $u$ is a solution, the corresponding scaling is formally given by
$$
u_\lambda(x,t)
=
\lambda^{2\gamma-1}
u(\lambda x,\lambda^{2\gamma}t),
\qquad \lambda>0.
$$
For the complete damped system, the same transformation also changes the damping coefficient. Indeed, the time derivative, the fractional dissipation, the transport term, and the pressure gradient scale with the factor $\lambda^{4\gamma-1}$, whereas
$$
[|u_\lambda|^2u_\lambda](x,t)
=
\lambda^{6\gamma-3}
[|u|^2u](\lambda x,\lambda^{2\gamma}t).
$$
Consequently, the transformed damping coefficient is
$$
\alpha_\lambda
=
\alpha\lambda^{2-2\gamma}.
$$
Hence, for fixed $\alpha>0$, the complete system is invariant under this scaling only when $\gamma=1$. When $\alpha=0$, on the other hand, the displayed transformation is a symmetry of the equation for every $\gamma\geq\frac12$.

For the Lei--Lin space
$$
\mathcal X^s(\mathbb R^3)
=
\left\{
f\in \cS'(\bR^3) : \hat{f}\in L^1_{loc}(\mathbb{R}^3) \hbox{  and  }
\int_{\mathbb R^3}
|\xi|^s|\widehat f(\xi)|\,d\xi<\infty
\right\},
$$
a direct computation gives
$$
\|u_\lambda(0)\|_{\mathcal X^s}
=
\lambda^{2\gamma-1+s}
\|u_0\|_{\mathcal X^s}.
$$
Consequently, the scaling-critical index is
$$
s_c=1-2\gamma.
$$
For a fixed Gevrey parameter $a>0$, the scaling behaves instead as
$$
\|u_\lambda(0)\|_{\mathcal X^s_{a,\sigma}}
=
\lambda^{2\gamma+s-1}
\|u_0\|_{\mathcal X^s_{a\lambda^{1/\sigma},\sigma}}.
$$
Therefore, a Gevrey norm with fixed radius $a>0$ is not itself invariant under the scaling. Throughout this work, the terminology ``critical'' refers to the underlying polynomial Lei--Lin index. In particular, $\mathcal X^0(\mathbb{R}^3)$ is critical when $\gamma=\frac12$, while $\mathcal X^0_{a,\sigma}(\mathbb{R}^3)$, with $a>0$, should be viewed as its Gevrey refinement with fixed radius.

As observed above, when $\gamma=\frac12$, one has $s_c=0$, and hence $\mathcal X^0(\mathbb R^3)$ is invariant under the natural scaling of the undamped equation. This gives a precise reason for the importance of $\mathcal X^0(\mathbb R^3)$ in the critical problem: rescaling the solution does not make the initial datum smaller in this norm. Thus, the critical theory can not rely on a gain produced simply by scaling. In this sense, $s=s_c$ marks the natural threshold between the subcritical regime $s>s_c$, where the norm becomes smaller under concentration, and the supercritical regime $s<s_c$, where the norm becomes larger.

The specific Lei--Lin--Gevrey space $\mathcal X^s_{a,\sigma}(\mathbb R^3)$ refines this Lei--Lin framework by introducing an exponential weight in frequency. More precisely, we consider
$$
\|f\|_{\mathcal X^s_{a,\sigma}}
=
\int_{\mathbb R^3}
|\xi|^s
e^{a|\xi|^{1/\sigma}}
|\widehat f(\xi)|\,d\xi,
$$
where $a>0$ and $\sigma\geq1$. The polynomial factor $|\xi|^s$ measures the usual Fourier regularity, whereas the exponential factor measures additional decay of the Fourier transform at high frequencies and hence Gevrey regularity of the solution. There are two reasons why these spaces are well suited to \eqref{NS}.  We refer to \cite{patricia,Bnovo,Bcritico,BNS,Bsubcritico,Nata,wilberclay26,MR1026858,LeiLin2011,artigowilberthyagomanasses,Nati,thyago7,coriolis} for related developments concerning Lei--Lin and Lei--Lin--Gevrey spaces.

Finally, it is convenient to formulate the problem in terms of mild solutions. Let $\{e^{-t(-\Delta)^\gamma}\}_{t \geq0}$ denote the fractional heat semigroup and let $\mathbb P$ be the Leray projector onto divergence-free vector fields. Applying $\mathbb P$ to the first equation of \eqref{NS}, we eliminate the pressure and obtain
$$
u_t+(-\Delta)^\gamma u
+
\mathbb P(u\cdot\nabla u)
+
\alpha\mathbb P(|u|^2u)
=0.
$$
Accordingly, a function $u$ is said to be a mild solution of \eqref{NS} on an interval $[0,T]$ if
$$
\begin{aligned}
u(t)
={}&
e^{-t(-\Delta)^\gamma}u_0
-
\int_0^t
e^{-(t-\tau)(-\Delta)^\gamma}
\mathbb P\bigl(u(\tau)\cdot\nabla u(\tau)\bigr)
\,d\tau
-
\alpha
\int_0^t
e^{-(t-\tau)(-\Delta)^\gamma}
\mathbb P\bigl(|u(\tau)|^2u(\tau)\bigr)
\,d\tau ,
\end{aligned}
$$
for every $t\in[0,T]$. Since $\operatorname{div}u=0$, one may equivalently write
$$
u\cdot\nabla u
=
\operatorname{div}(u\otimes u),
$$
so that the quadratic contribution is given by
$$
\mathbb P[\operatorname{div}(u\otimes u)].
$$
The mild formulation is especially useful at the regularity considered here, since it avoids interpreting each derivative in the classical pointwise sense. The linear dissipation is completely encoded in the semigroup $S_\gamma(t)$, while the nonlinear terms are handled through Duhamel integrals. The smoothing and decay of $S_\gamma(t)$ can therefore be used directly against the Fourier convolution structure of the nonlinearities. This is the basic mechanism behind the existence, blow-up, and stability results obtained below.

We now present the main results of the paper.  Our first result provides the basic well-posedness theory required throughout the paper. It treats both the critical case $\gamma=\frac12$ and the subcritical case $\gamma>\frac12$, and it also separates the effect of the damping term. In the critical undamped problem, small initial data generate global mild solutions. In the critical damped problem, we obtain a local theory for sufficiently small data. In the subcritical regime, the additional strength of the fractional dissipation allows us to construct a unique local mild solution for arbitrary data in $\cX^0_{a,\sigma}(\bR^3)$.

\begin{theorem} \label{thm:solucaolocal}
Let $a \geq 0$, $\sigma \geq 1$ and $u_0 \in \cX^0_{a,\sigma}(\bR^3)$ such that $\hbox{div}\,u_0=0$. Then,
\begin{enumerate}
  \item[\emph{i)}] For $\gamma =\frac{1}{2}$ and $\alpha=0$, there exists a constant $C'>0$ such that, if $\|u_0\|_{\cX^0_{a,\sigma}}<C'$, then there is a unique solution
	$$
		u \in C_T(\cX^{0}_{a,\sigma}(\bR^3)) \cap L^1_T(\cX^{1}_{a,\sigma}(\bR^3))
	$$
	to the Navier--Stokes equations \emph{(\ref{NS})}, for every $T > 0$. Moreover,
\begin{align}\label{w28w}
\Vert u (t) \Vert_{\cX^0_{a,\sigma}} \leq 2 \Vert u_0 \Vert_{\cX^{0}_{a,\sigma}},
\quad\forall t\in [0,T].
\end{align}
If
\begin{align*}
u \in C([0,T^*), \cX^0_{a,\sigma}(\bR^3))
\cap
L^1_{\mathrm{loc}}([0,T^*), \cX^1_{a,\sigma}(\bR^3))
\end{align*}
is the maximal solution to the Navier--Stokes equations \emph{(\ref{NS})} with $\|u_0\|_{\cX^0_{a,\sigma}}<C'$, then $T^*=\infty$ and
\begin{align}\label{w30}
u \in C([0,\infty), \cX^0_{a,\sigma}(\bR^3))
\cap
L^1([0,\infty), \cX^1_{a,\sigma}(\bR^3)).
\end{align}

\item[\emph{ii)}] For $\gamma =\frac{1}{2}$ and $\alpha>0$, there exist a time $\bar T=\bar T (\alpha,u_0)>0$ and a constant $C'>0$ such that, if $\|u_0\|_{\cX^0_{a,\sigma}}<C'$, then there is a unique solution
$$
u \in C_{\bar T}(\cX^{0}_{a,\sigma}(\bR^3))
\cap
L^1_{\bar T}(\cX^{1}_{a,\sigma}(\bR^3))
$$
to the Navier--Stokes equations \emph{(\ref{NS})}. Moreover,
\begin{align}\label{w28}
\Vert u (t) \Vert_{\cX^0_{a,\sigma}}
\leq
2 \Vert u_0 \Vert_{\cX^{0}_{a,\sigma}},
\quad\forall t\in [0,\bar T].
\end{align}

\item[\emph{iii)}] For $\gamma>\frac{1}{2}$ and $\alpha\geq 0$, there exist a time $\bar T=\bar T(\alpha,\gamma,u_0)>0$ and a unique solution
$$
u \in C_{\bar T}(\cX^{0}_{a,\sigma}(\bR^3))
\cap
L^1_{\bar T}(\cX^{2\gamma}_{a,\sigma}(\bR^3))
$$
to the Navier--Stokes equations \emph{(\ref{NS})}. Moreover,
\begin{align*}
\Vert u (t) \Vert_{\cX^0_{a,\sigma}}
\leq
2 \Vert u_0 \Vert_{\cX^{0}_{a,\sigma}},
\quad\forall t\in [0,\bar T].
\end{align*}
\end{enumerate}
\end{theorem}

A result corresponding to Theorem \ref{thm:solucaolocal} i) has recently been obtained by W. G. Melo, N. F. Rocha and N. dos S. Costa \cite{Nati}. Their proof, however, follows a different route from the one used here; see, in particular, Lemmas 3.1 and 3.2 of \cite{Nati}. We also recall that global existence for subcritical dissipative quasi-geostrophic and Navier--Stokes equations, namely $\frac12<\gamma\leq1$, has been established in the scaling-critical Lei--Lin space $\cX^{1-2\gamma}(\bR^3)$; see \cite{Bsubcritico,BNS} and the references therein. In addition, W. G. Melo, M. de Souza and T. S. R. Santos \cite{artigowilberthyagomanasses} obtained a global solution to \eqref{NS}, with $\gamma=\frac12$ and $\alpha=0$, in the class $C_b([0,\infty),X^0_{a,\sigma}(\mathbb{R}^3))$. We emphasize that our proof does not use the inequality (40) from \cite{artigowilberthyagomanasses}. Finally, J. Benameur and L. Jlali \cite{Bnovo} considered the case $\gamma=1$ and $\alpha>0$ in the noncritical Lei--Lin space $\cX^0(\bR^3)$; see Theorem 1.1 of \cite{Bnovo}. These works provide useful points of comparison, while the formulation above gives the common well-posedness framework used in the remainder of the present paper.

Our next result addresses the possible breakdown of the subcritical mild solution obtained in Theorem \ref{thm:solucaolocal} iii). The point is not only to show that the $\cX^0_{a,\sigma}(\mathbb{R}^3)$-norm must become unbounded if the maximal existence time is finite, but also to identify a nonintegrable quantity and a quantitative lower bound for its growth. Thus, finite-time blow-up, if it occurs, can not happen arbitrarily slowly.

\begin{theorem} \label{thm:blowup}
Let $a \geqslant 0$, $\sigma \geqslant 1$, $\gamma > \frac{1}{2}$, $\alpha\geq0$ and $u_0 \in \cX^0_{a,\sigma}(\bR^3)$ such that $\hbox{div}\,u_0=0$. Assume that
$$
u \in C([0,T^*), \cX^{0}_{a,\sigma}(\bR^3))
\cap
L^1_{\mathrm{loc}}([0,T^*), \cX^{2\gamma}_{a,\sigma}(\bR^3))
$$
is the maximal solution to the Navier--Stokes equations \emph{(\ref{NS})} obtained in Theorem \emph{\ref{thm:solucaolocal}} \emph{iii)}.
If $T^*<\infty$, then
\begin{enumerate}
  \item[\emph{i)}] $\displaystyle \limsup_{t \nearrow T^*} \Vert u(t) \Vert_{\cX^{0}_{a,\sigma}} = \infty$;
  \item[\emph{ii)}] $\displaystyle \int_t^{T^*} \left[\Vert u(\tau) \Vert_{\cX^{0}_{a,\sigma}}^{\frac{2\gamma}{2\gamma - 1}} + \Vert u(\tau) \Vert_{\cX^{0}_{a,\sigma}}^2\right] \,d\tau = \infty$;
  \item[\emph{iii)}] $\displaystyle \frac{[C_\gamma (\alpha + 1)]^{-1}}{T^* - t} \leq \Vert u(t) \Vert_{\cX^{0}_{a,\sigma}}^{\frac{2\gamma}{2\gamma - 1}} + \Vert u(t) \Vert_{\cX^{0}_{a,\sigma}}^2$,
\end{enumerate}
for every $t\in [0,T^*)$, where $C_\gamma>0$ depends only on $\gamma$.
\end{theorem}

For the undamped problem $\alpha=0$, the same argument gives sharper versions of parts ii) and iii). More precisely, they can be replaced by
\begin{enumerate}
  \item [ii')] $\displaystyle \int_t^{T^*} \Vert u(\tau) \Vert_{\cX^{0}_{a,\sigma}}^{\frac{2\gamma}{2\gamma - 1}}\,d\tau = \infty$;
  \item[iii')] $\displaystyle \frac{C_\gamma }{(T^* - t)^{\frac{2\gamma-1}{2\gamma}}} \leq \Vert u(t) \Vert_{\cX^{0}_{a,\sigma}},$
\end{enumerate}
for every $t\in [0,T^*)$. For details, see inequality (\ref{w51}) below and repeat the argument with $\alpha=0$. These conclusions were also obtained in \cite{Nata}; see Theorems 1.3 ii) and iii) therein. We further note that J. Benameur and L. Jlali \cite{Bnovo} studied the case $\gamma=1$ and $\alpha>0$ in the noncritical Lei--Lin space $\cX^0(\bR^3)$; see Theorem 1.2 of \cite{Bnovo}. The result above extends the discussion to fractional dissipation and, importantly for our purposes, keeps track of the Gevrey radius.

Motivated by P. L. Guidolin, W. G. Melo and T. S. R. Santos \cite{patricia}, we next refine the blow-up analysis by comparing norms with different Gevrey radii. This gives a family of necessary conditions for finite-time breakdown and shows that blow-up in the strongest Gevrey norm forces quantitative growth even after part of the exponential Fourier weight is removed. More precisely, we have the following result.

\begin{corollary}\label{corollaryB1}
Assume that $a\geq0$, $\sigma\geq1$, $\gamma>\frac{1}{2}$, $\alpha\geq0$ and $u_0\in \mathcal{X}_{a,\sigma}^0(\mathbb{R}^3)$ such that $\hbox{div}\,u_0=0$.
Assume that $u\in C([0,T^*),\mathcal{X}_{a,\sigma}^{0}(\mathbb{R}^3))$
is a maximal solution to the Navier--Stokes equations \emph{(\ref{NS})} obtained in Theorem \emph{\ref{thm:solucaolocal} iii)}. If $T^*<\infty$, then
\begin{enumerate}
\item[\emph{i)}] $\displaystyle\int_t^{T^*}\Big[\|u(\tau)\|_{\mathcal{X}_{\frac{a}{\sigma(\sqrt{\sigma})^{n-1}},\sigma}^{0}}^{\frac{2\gamma}{2\gamma-1}} + \|u(\tau)\|_{\mathcal{X}_{\frac{a}{\sigma(\sqrt{\sigma})^{n-1}},\sigma}^{0}}^{2}\Big]\;d\tau=\infty$;
\item[\emph{ii)}] $\displaystyle \|u(t)\|_{\mathcal{X}_{\frac{a}{(\sqrt{\sigma})^{n-1}},\sigma}^{0}}^{\frac{2\gamma}{2\gamma-1}} + \|u(t)\|_{\mathcal{X}_{\frac{a}{(\sqrt{\sigma})^{n-1}},\sigma}^{0}}^{2}\geq  \frac{[C_{\gamma}(\alpha+1)]^{-1}}{T^*-t}$;
\item[\emph{iii)}] $\displaystyle \limsup_{t\nearrow T^*} \|u(t)\|_{\mathcal{X}_{\frac{a}{(\sqrt{\sigma})^{n-1}},\sigma}^{0}}=\infty$;
\item[\emph{iv)}] $\displaystyle \|u(t)\|_{\mathcal{X}^0}^{\frac{2\gamma}{2\gamma-1}}+ \|u(t)\|_{\mathcal{X}^0}^{2}\geq \frac{[C_\gamma(\alpha+1)]^{-1}}{T^{*}-t}$, provided that $\sigma>1$.
\end{enumerate}
for every $t\in[0,T^*)$ and $n\in \mathbb{N}$, where $C_\gamma>0$ depends only on $\gamma$.
\end{corollary}

Still following the ideas in \cite{patricia}, we return to the case $\gamma=1$ and $\alpha>0$, considered by J. Benameur and L. Jlali \cite{Bnovo} in the particular Lei--Lin space $\cX^0(\bR^3)$. The Gevrey structure allows us to obtain substantially more information on the possible growth of a solution near its maximal existence time. Indeed, by combining Corollary \ref{corollaryB1} iv) with inequality (\ref{lemanovo2}) below, we obtain an exponential lower bound for the $\cX^0_{a,\sigma}(\mathbb{R}^3)$-norm. This is one of the main reasons for working in the strict Lei--Lin--Gevrey space with $a>0$: the exponential Fourier weight detects a blow-up mechanism that is not visible at the level of the polynomial Lei--Lin norm alone.

\begin{corollary}\label{corollaryB2}
Assume that $a>0$, $\sigma>1$, $\gamma=1$, $\alpha\geq0$ and $u_0\in \mathcal{X}_{a,\sigma}^0(\mathbb{R}^3)\cap L^2(\mathbb{R}^3)$ such that $\hbox{div}\,u_0=0$.
Assume that $u\in C([0,T^*),\mathcal{X}_{a,\sigma}^{0}(\mathbb{R}^3))$
is a maximal solution to the Navier--Stokes equations \emph{(\ref{NS})} obtained in Theorem \emph{\ref{thm:solucaolocal} iii)}. If $T^*<\infty$, then
$$
\|u(t)\|_{\mathcal{X}_{a,\sigma}^{0}}
\geq
C(\alpha+1)^{-\frac{1}{2}}(T^*-t)^{-\frac{1}{2}}
\exp\left\{
a\left[
C(\alpha+1)^{-\frac{1}{2}}
(T^*-t)^{-\frac{1}{2}}
\|u_0\|_{L^2}^{-1}
\right]^{\frac{2}{3\sigma}}
\right\},
$$
for every $t\in[0,T^*)$, where $C>0$ is a constant.
\end{corollary}

We conclude the introduction with a stability result for global mild solutions of the critical dissipative Navier--Stokes equations without damping, that is, $\gamma=\frac12$ and $\alpha=0$. Stability results for the subcritical case $\gamma>\frac12$ have been studied in several settings; see, for instance, \cite{BNS,Robert}. Here, we show that the small global solution constructed in Theorem \ref{thm:solucaolocal} i) is stable under sufficiently small perturbations of its initial datum in the same critical Lei--Lin--Gevrey space. More precisely, let $v_0\in\cX^{0}_{a,\sigma}(\bR^3)$ be divergence free and sufficiently close to $u_0$. We consider the corresponding system
\begin{equation}\label{NSv}	\tag{NS$v$}
\left\{
\begin{array}{l}
v_t
\;\!+\,
(-\Delta)^{\frac{1}{2}}\,v
\,+\,
v \cdot \nabla v
\,+\,
\nabla \;\!q \:\!
\;=\;
0, \quad x\in\mathbb{R}^3, t>0;\\
\mbox{div}\:v  \;=\; 0, \quad x\in\mathbb{R}^3, t>0;\\
v(x,0) \,=\, v_0(x), \quad x\in\mathbb{R}^3,
\end{array}
\right.
\end{equation}
where $q$ denotes the pressure associated with $v$. We prove that the solution $v$ is also global and belongs to the same class as $u$, and, more importantly, that the distance between the two solutions remains quantitatively controlled for all positive times. Thus, a small perturbation of a small global critical solution does not produce a different long-time regime. The uniform estimate (\ref{w28w}) is essential in closing the perturbative argument; see (\ref{w6}) and (\ref{w7}) below.

\begin{theorem} \label{thm:solucaoestabilidade}
Let $a \geq 0$, $\sigma \geq 1$, $\gamma=\frac{1}{2}$, $\alpha=0$ and $u_0 \in \cX^0_{a,\sigma}(\bR^3)$ such that $\hbox{div}\,u_0=0$. Suppose that
$$
\Vert u_0 \Vert_{\cX^{0}_{a,\sigma}}< \frac{C'}{2},
$$
where $C'$ is the constant given in Theorem \emph{\ref{thm:solucaolocal} i)}, and assume that
\begin{align}\label{w1}
u \in C([0,\infty), \cX^0_{a,\sigma}(\bR^3))
\cap
L^1([0,\infty), \cX^1_{a,\sigma}(\bR^3))
\end{align}
is the global solution to the Navier--Stokes equations \emph{(\ref{NS})} obtained in Theorem \emph{\ref{thm:solucaolocal} i)}.
If $v_0 \in \cX^{0}_{a,\sigma}(\bR^3)$ satisfies $\hbox{div}\,v_0=0$ and
\begin{equation} \label{eq:H1}
\Vert u_0 - v_0 \Vert_{\cX^{0}_{a,\sigma}}
<
\frac{C'}{2}
\exp \left(
-C''
\int_0^\infty
\Vert u(\tau) \Vert_{\cX^{1}_{a,\sigma}}
\,d\tau
\right),
\end{equation}
for an appropriate constant $C''>0$, then the Navier--Stokes equations \emph{(\ref{NSv})} admit a unique global solution
$$
v \in C([0,\infty), \cX^{0}_{a,\sigma}(\bR^3))
\cap
L^1([0,\infty), \cX^1_{a,\sigma}(\bR^3))
$$
satisfying
\begin{align}\label{estabilidade}
\Vert u(t) - v(t) \Vert_{\cX^{0}_{a,\sigma}}
+
\frac{1}{2}
\int_0^t
\Vert u(\tau) - v(\tau) \Vert_{\cX^{1}_{a,\sigma}}
\,d\tau
\leq
\Vert u_0 - v_0 \Vert_{\cX^{0}_{a,\sigma}}
\exp \left(
C''
\int_0^\infty
\Vert u(\tau) \Vert_{\cX^{1}_{a,\sigma}}
\,d\tau
\right),
\end{align}
for every $t \geq 0$.
\end{theorem}

Let us finally recall that stability results related to Theorem \ref{thm:solucaoestabilidade} have been studied for the subcritical Navier--Stokes equations, namely $\gamma>\frac12$, in several functional settings, including Sobolev--Gevrey spaces; see \cite{Robert} and the references therein. The result above addresses instead the borderline dissipation $\gamma=\frac12$ directly in the critical Lei--Lin--Gevrey framework and gives a global quantitative estimate for the propagation of perturbations.

\begin{remark}
By observing the techniques presented in this paper, it is interesting to notice that we also have a technical  explanation for our  choice of the specific Lei--Lin--Gevrey space $\cX^{0}_{a,\sigma}(\mathbb{R}^3)$. In fact, if we reevaluate the arguments used in the proof of Theorem \ref{thm:solucaolocal}, with the space $\cX^{s}_{a,\sigma}(\mathbb{R}^3)$ instead, we should estimate, for example, the terms 
\begin{align*}
	\|[(\phi\cdot \psi)\theta]\|_{\cX^{s}_{a,\sigma}} \hbox{  and  } \|v\|_{\cX^{0}_{a,\sigma}},
\end{align*}
%by 
%\begin{align*}
%	\|\phi\|_{\cX^{s}_{a,\sigma}}\|\psi\|_{\cX^{s}_{a,\sigma}}\|\theta\|_{\cX^{s}_{a,\sigma}} \hbox{  and  } %\|v\|_{\cX^{s}_{a,\sigma}}^{1+\frac{s}{2\gamma}}\|v\|_{\cX^{s+2\gamma}_{a,\sigma}}^{-\frac{s}{2\gamma}}.
%\end{align*}
see (\ref{w181}) and (\ref{wfinal1}) below. However, because of (\ref{lem:produto}) and Lemma 2.5 ii) in \cite{Nata}, we would be obligated to assume $s\geq0$ and $s\leq0 $, respectively.

\end{remark}

%\noindent \textbf{Acknowledgments:}
%\subsection*{Acknowledgments}
%The author T.S.R. Santos is partially supported by FAPESP grant 2024/15587-1.

\section{Notations and definition}
	The most important notations, definitions and results used in this paper are given below:
	\begin{itemize}
	%	\item The $i$th partial derivative is denoted by $\partial_i = \partial /\partial x_i$ (for $i = 1,2,3$). Likewise, the partial derivative with respect to time $t$, will be denoted as $\partial_t = \partial/\partial t$.
		\item The Fourier transform of $f\in S'(\mathbb{R}^3)$ (space of tempered distributions) is given by
		$$
			\cF[f](\xi) = \hat f (\xi) := \int_{\bR^3}e^{-ix \cdot \xi} f(x) \,dx,\quad\forall \xi \in \bR^3.
		$$
%				\item Similarly, the inverse Fourier transform of $f$ is given by
%		$$
%			\cF^{-1}[f](x) = \check f(x) := \int_{\bR^3} e^{ix \cdot \xi} f(\xi) \,d\xi,\quad\forall x \in \bR^3.
%		$$
		\item The Lei-Lin-Gevrey space $\cX^s_{a,\sigma}(\bR^3)$, with $a \geq 0$, $\sigma \geq 1$ and $s\in \mathbb{R}$, is given by
		$$
			\cX^s_{a,\sigma}(\bR^3) = \{f \in \cS'(\bR^3) : \hat{f}\in L^1_{loc}(\mathbb{R}^3)\hbox{ and } \Vert f \Vert_{\cX^s_{a,\sigma}} < \infty\},
		$$
		where
		$$
			\Vert f \Vert_{\cX^s_{a,\sigma}} = \int_{\bR^3} |\xi|^s e^{a|\xi|^{\frac{1}{\sigma}}} |\hat f(\xi)| \,d\xi.
		$$
		 In addition, $\cX^s_{a,\sigma}(\bR^3)$ becomes  Lei-Lin space $\cX^s(\bR^3)$ if $a=0$, and it is important to highlight that $\cX^s_{a,\sigma}(\bR^3)$ ($a\geq0$) is a Banach space provided that $s\leq0$.
		\item %The fractional Laplacian $(-\Delta)^\gamma$, for $\gamma > 0$, is defined by
		%$$
		%	\cF [(-\Delta)^\gamma f](\xi) = |\xi|^{2\gamma} \hat f(\xi),\quad\forall \xi\in \mathbb{R}^3,
		%$$		and 
The fractional heat semigroup is given by
		$$
			\cF [e^{-t (-\Delta)^\gamma}f](\xi) = e^{-t|\xi|^{2\gamma}} \hat f(\xi),\quad\forall \xi\in \mathbb{R}^3,t\geq0,
		$$
		for all $f\in S'(\mathbb{R}^3)$. In this paper, we shall work in  the case $\gamma \geq \frac{1}{2}$.
		\item The tensor product between two functions $f,g\in S'(\mathbb{R}^3)$ is given by
		$$
			f \otimes g = (g_1 f, g_2f, g_3 f).
		$$
		\item Assume that $(X, \Vert \cdot \Vert_X)$ is a normed space and $T > 0$.
		The space $L^p_T(X)$, for $1 \leq p \leq \infty$, contains all measurable functions $f : [0,T] \to X$ such that
		$$
			\Vert f \Vert_{L^p_T(X)} = \left( \int_0^T \Vert f(t) \Vert_X^p \,dt \right)^{\frac{1}{p}}<\infty,
		$$
		if $1 \leq p < \infty$, and also
		$$
			\Vert f \Vert_{L^{\infty}_T(X)} = \esssup_{t \in [0,T]} \Vert f(t) \Vert_X<\infty,
		$$
		if $p = \infty$.
		Furthermore, the space $C_T(X)$ of continuous functions $f : [0,T] \to X$ will be endowed with the norm $\Vert \cdot \Vert_{L^{\infty}_T(X)}$.
		\item Constants denoted by $C$ may change only in estimates in which no numerical threshold is subsequently defined.  In the fixed-point argument, $K_2$ and $K_3$ denote fixed positive constants for the bilinear and trilinear estimates, respectively; every smallness and lifespan condition below is stated in terms of these frozen constants.  As usual, $C_r$ depends only on $r$.
\item Let $a \geq 0$, $\sigma \geq 1$ and $s\geq -1$, then 
	\begin{align}\label{lem:produto}
		\nonumber \Vert fg \Vert_{\cX^{s+1}_{a,\sigma}} &\leq C \big[\Vert f \Vert_{\cX^0_{\frac{a}{\sigma},\sigma}} \Vert g \Vert_{\cX^{s+1}_{a,\sigma}} + \Vert f \Vert_{\cX^{s+1}_{a,\sigma}} \Vert g \Vert_{\cX^0_{\frac{a}{\sigma},\sigma}} \big]\\
&\leq C \big[\Vert f \Vert_{\cX^0_{a,\sigma}} \Vert g \Vert_{\cX^{s+1}_{a,\sigma}} + \Vert f \Vert_{\cX^{s+1}_{a,\sigma}} \Vert g \Vert_{\cX^0_{a,\sigma}} \big],
	\end{align}
	where $C = 2^{s-2} \pi^{-3}$. 
The proof of this inequality above is given by \cite{Nata} (see Lemma 2.6 in \cite{Nata}).
\item 	Let $a \geqslant 0$, $\sigma \geqslant 1$, and $\gamma \geqslant \frac{1}{2}$. Then, it holds
	\begin{align}\label{lem:split}
		\Vert f \Vert_{\cX^{1}_{a,\sigma}} \leq \Vert f \Vert_{\cX^{0}_{a,\sigma}}^{1-\frac{1}{2\gamma}} \Vert f \Vert^{\frac{1}{2\gamma}}_{\cX^{2\gamma}_{a,\sigma}}.
	\end{align}
The proof of this inequality above is given by \cite{Nata} (see Lemma 2.5 i) in \cite{Nata}).
%\item 	Let $a>0$, $\sigma\geq 1$, $\mu>1$, $s\leq0$, $\delta\in\mathbb{R}$ and $f\in \mathcal{X}_{a,\sigma}^{s+\delta}(\mathbb{R}^3)$. Then, there exists a positive constant %$C_{a,s,\delta,\sigma,\mu}$ such that
%	\begin{align}\label{lema0}
%	\|f\|_{\mathcal{X}_{\frac{a}{\mu},\sigma}^{\delta}}\leq C_{a,s,\delta,\sigma,\mu}\|f\|_{\mathcal{X}_{a,\sigma}^{s+\delta}}.
%	\end{align}
%The proof of this inequality above is given by \cite{Nata} (see Lemma 2.4 in \cite{Nata}).
\item 	Let $\delta>0$; then, the following inequality holds:
	\begin{align}\label{lema}
	\|f\|_{\mathcal{X}^0}\leq C \|f\|_{L^2}^{\frac{2\delta}{2\delta+3}}\|f\|_{\mathcal{X}^\delta}^{\frac{3}{2\delta+3}}.
	\end{align}
For more details, see Lemma 2.2 in \cite{Nata}.
\item 	Let  $a>0$, $\sigma\geq1$ and  $\delta\geq0$. Then,  there is a positive constant $C_{a,s,\delta,\sigma}$ such that
	\begin{align}\label{lemanovo2}
	\|f\|_{\mathcal{X}^{\delta}}\leq C_{a,\delta,\sigma}\|f\|_{\mathcal{X}_{a,\sigma}^{0}}.
	\end{align}
For more details, see Lemma 2.3 in \cite{Nata}.
	\end{itemize}

\section{Proof of our results} \label{sec:demonstracao}

In this section, we shall present the proof of our main results, which ones are given by Theorems \ref{thm:solucaolocal}, \ref{thm:blowup} and \ref{thm:solucaoestabilidade}, and Corollaries \ref{corollaryB1} and \ref{corollaryB2}. It is important to point out that we shall adapt and improve some of the arguments established by \cite{patricia,Bnovo,Nata,Robert} (see references therein as well). Let us prove Theorem \ref{thm:solucaolocal} firstly.

\bigskip
\noindent\textbf{Proof of Theorem \ref{thm:solucaolocal}:}
Let $T > 0$ (to be established as follows) and define the Banach space
$$
	\mathcal{X}_T = C_T(\cX^{0}_{a,\sigma}(\bR^3)) \cap L^1_T(\cX^{2\gamma}_{a,\sigma}(\bR^3))
$$
%\revision{Sempre tenho minhas duvidas nessa parte, pq $\mathcal{X}^s$ não é Banach para $s\geq0$. O certo não seria considerar $\mathcal{Y}^s= \mathcal{X}^s \cap \mathcal{X}^0$? pois esse é sempre Banach para todo $s\geq0$. Ou entao fornecer uma prova rapida de que esse espaço é Banach} 
endowed with the norm
$$
	\Vert u \Vert_T = \Vert u \Vert_{L^\infty_T(\cX^{0}_{a,\sigma})} + \Vert u \Vert_{L^1_T(\cX^{2\gamma}_{a,\sigma})},\quad\forall u\in \mathcal{X}_T.
$$
For $\gamma\geq\frac{1}{2}$ and $R > 0$ (to be chosen below), define the closed
subset $\mathcal{F}_{T}$ of $\mathcal{X}_T$ by
$$
	\mathcal{F}_{T} = \{ u \in \mathcal{X}_T : \Vert u \Vert_{L^{\infty}_T(\cX^0_{a,\sigma})} \leq  2\Vert u_0 \Vert_{\cX^{0}_{a,\sigma}} \text{ and } \Vert u \Vert_{L^1_T(\cX^{2\gamma}_{a,\sigma})} \leq R\}.
$$
Now, consider the operator $\Psi : \mathcal{F}_{T} \to \mathcal{X}_T$ given by
\begin{align}\label{w24}
	\Psi[u](t) = e^{-t(-\Delta)^{\gamma}} u_0 - B(u,u)(t)-\alpha B'(u,u,u)(t), \quad\forall u\in \mathcal{F}_{T},t\in [0,T],
\end{align}
where
\begin{align}\label{w52}
	B(w,v)(t) = \int_0^t e^{-(t-\tau)(-\Delta)^{\gamma}} \rP(w \cdot \nabla v)(\tau) \,d\tau,\quad\forall w,v\in\mathcal{F}_{T},t\in[0,T],
\end{align}
and also
\begin{align}\label{w53}
		B'(\phi,\psi,\theta)(t) = \int_0^t e^{-(t-\tau)(-\Delta)^{\gamma}} \rP [(\phi \cdot \psi)\theta](\tau) \,d\tau ,\quad\forall \phi,\psi,\theta\in\mathcal{F}_{T},t\in[0,T],
	\end{align}
with $\hbox{div}\, w = \hbox{div}\, v= \hbox{div}\,\phi=\hbox{div}\,\psi=\hbox{div}\,\theta=0$. Our goal is to apply  Banach Fixed Point Theorem to the operator $\Psi$ in order to show the existence and uniqueness of  solutions for the Navier-Stokes equations (\ref{NS}). Hence, it is enough to prove that
\begin{gather}
	\Psi(\mathcal{F}_{T}) \subseteq \mathcal{F}_{T} \label{eq:claim1} \,\hbox{ and also}\\
	\Vert \Psi[u] - \Psi[v] \Vert_T \leq \frac{1}{2} \Vert u - v \Vert_T, \quad \forall u, v \in \mathcal{F}_{T}. \label{eq:claim2}
\end{gather}

Let us start proving the inclusion presented in (\ref{eq:claim1}). This means that we shall show that $\Psi[u]\in \mathcal{F}_{T}$, for all $u\in \mathcal{F}_{T}$.
%$$
%	\Vert \Psi[u] \Vert_{L^{\infty}_T(\cX^{0}_{a,\sigma})} \leq 2 \Vert u_0 \Vert_{\cX^{0}_{a,\sigma}}
%$$
%and
%$$
	%\Vert \Psi[u] \Vert_{L^1_T(\cX^{1}_{a,\sigma})} \leq R,
%$$
%for any $u \in F_{RT}$.
Indeed, it is true that
\begin{align*}
	\Vert e^{-t(-\Delta)^\gamma} u_0 \Vert_{\cX^{0}_{a,\sigma}} &=\int_{\mathbb{R}^3} e^{a|\xi|^{\frac{1}{\sigma}}} e^{-t|\xi|^{2\gamma}}|\widehat{u_0}(\xi)|\,d\xi\leq\int_{\mathbb{R}^3} e^{a|\xi|^{\frac{1}{\sigma}}} |\widehat{u_0}(\xi)|\,d\xi=\Vert  u_0 \Vert_{\cX^{0}_{a,\sigma}},
\end{align*}
for all $t\in [0,T]$ and $\gamma \geq \frac{1}{2}$. Thereby, we can write the following inequality:
\begin{equation} \label{eq:A1}
	\Vert e^{-t(-\Delta)^\gamma} u_0 \Vert_{L^{\infty}_T(\cX^{0}_{a,\sigma})} \leq \Vert u_0 \Vert_{\cX^{0}_{a,\sigma}},
\end{equation}
%for every $T > 0$.
for $\gamma \geq \frac{1}{2}$. In addition, we can observe that the equalities below are true:
\begin{align}\label{w35}
	\nonumber\Vert e^{-t(-\Delta)^\gamma} u_0 \Vert_{L^{1}_T(\cX^{2\gamma}_{a,\sigma})} &=\int_0^T\int_{\mathbb{R}^3} |\xi|^{2\gamma}e^{a|\xi|^{\frac{1}{\sigma}}} e^{-t|\xi|^{2\gamma}}|\widehat{u_0}(\xi)|\,d\xi dt\\
\nonumber&=\int_{\mathbb{R}^3}|\xi|^{2\gamma} e^{a|\xi|^{\frac{1}{\sigma}}} |\widehat{u_0}(\xi)|\left(\int_0^Te^{-t|\xi|^{2\gamma}}\,dt\right)\,d\xi\\
&=\int_{\mathbb{R}^3} [1-e^{-T|\xi|^{2\gamma}}]e^{a|\xi|^{\frac{1}{\sigma}}} |\widehat{u_0}(\xi)|\,d\xi,
\end{align}
for $\gamma \geq\frac{1}{2}$. As a result, we deduce
\begin{align}\label{w16}
	\Vert e^{-t(-\Delta)^\gamma} u_0 \Vert_{L^{1}_T(\cX^{2\gamma}_{a,\sigma})}&\leq\Vert  u_0 \Vert_{\cX^{0}_{a,\sigma}}.
\end{align}
%Thus, by  (\ref{eq:A1}) and (\ref{w16}), it follows that
%\begin{equation*}
%	\Vert e^{-t(-\Delta)^\frac{1}{2}} u_0 \Vert_{T} \leq 2\Vert u_0 \Vert_{\cX^{0}_{a,\sigma}}.
%\end{equation*}
On the other hand, by the definition (\ref{w52}),  we obtain
\begin{align}\label{w18} 
	\nonumber \Vert B(w,v) (t)\Vert_{\cX^{0}_{a,\sigma}} &\leq \int_0^t \|e^{-(t-\tau)(-\Delta)^{\gamma}}\mathbb{P}(w\cdot \nabla v)(\tau)\|_{\cX^{0}_{a,\sigma}}\,d\tau\\
\nonumber&=\int_0^t \int_{\mathbb{R}^3}e^{a|\xi|^{\frac{1}{\sigma}}}e^{-(t-\tau)|\xi|^{2\gamma}}|\mathcal{F}[\mathbb{P}(w\cdot \nabla v)](\xi,\tau)|\,d\xi d\tau\\
\nonumber&\leq\int_0^t \int_{\mathbb{R}^3}|\xi|e^{a|\xi|^{\frac{1}{\sigma}}}|\mathcal{F}[v\otimes w](\xi,\tau)|\,d\xi d\tau\\
&=\int_0^t \|(v\otimes w)(\tau)\|_{\cX^{1}_{a,\sigma}} d\tau,
\end{align}
for all $t\in[0,T]$. Then,  (\ref{lem:produto}), (\ref{lem:split}) and Hölder's inequality imply that
\begin{align} \label{wfinal1}
	\nonumber \Vert B(w,v) (t)\Vert_{\cX^{0}_{a,\sigma}} &\leq  C\int_0^T [\|v\|_{\cX^{0}_{a,\sigma}}\|w\|_{\cX^{1}_{a,\sigma}}+\|v\|_{\cX^{1}_{a,\sigma}}\|w\|_{\cX^{0}_{a,\sigma}}] d\tau\\
&\leq  C\int_0^T [\|v\|_{\cX^{0}_{a,\sigma}} \Vert w \Vert_{\cX^{0}_{a,\sigma}}^{1-\frac{1}{2\gamma}} \Vert w \Vert^{\frac{1}{2\gamma}}_{\cX^{2\gamma}_{a,\sigma}}  +\Vert v \Vert_{\cX^{0}_{a,\sigma}}^{1-\frac{1}{2\gamma}} \Vert v \Vert^{\frac{1}{2\gamma}}_{\cX^{2\gamma}_{a,\sigma}}\|w\|_{\cX^{0}_{a,\sigma}}] d\tau\\
\nonumber &\leq  C[\|v\|_{L^{\infty}_T(\cX^{0}_{a,\sigma})}\|w\|_{L^{\infty}_T(\cX^{0}_{a,\sigma})}^{1-\frac{1}{2\gamma}} \int_0^T \Vert w \Vert^{\frac{1}{2\gamma}}_{\cX^{2\gamma}_{a,\sigma}}d\tau +\|w\|_{L^{\infty}_T(\cX^{0}_{a,\sigma})}\|v\|_{L^{\infty}_T(\cX^{0}_{a,\sigma})}^{1-\frac{1}{2\gamma}}\int_0^T \Vert v \Vert^{\frac{1}{2\gamma}}_{\cX^{2\gamma}_{a,\sigma}}d\tau]\\
\nonumber&\leq  CT^{1-\frac{1}{2\gamma}}[\|v\|_{L^{\infty}_T(\cX^{0}_{a,\sigma})}\|w\|_{L^{\infty}_T(\cX^{0}_{a,\sigma})}^{1-\frac{1}{2\gamma}} \|w\|_{L^{1}_T(\cX^{2\gamma}_{a,\sigma})}^{\frac{1}{2\gamma}} +\|w\|_{L^{\infty}_T(\cX^{0}_{a,\sigma})}\|v\|_{L^{\infty}_T(\cX^{0}_{a,\sigma})}^{1-\frac{1}{2\gamma}} \|v\|_{L^{1}_T(\cX^{2\gamma}_{a,\sigma})}^{\frac{1}{2\gamma}}],
\end{align}
for $\gamma \geq\frac{1}{2}$ and for all $t\in[0,T]$. Hence, it is true that  
\begin{align}\label{w17} 
	\nonumber \Vert B(w,v) \Vert_{L^\infty_T(\cX^{0}_{a,\sigma})} &\leq CT^{1-\frac{1}{2\gamma}}\|v\|_{L^{\infty}_T(\cX^{0}_{a,\sigma})}\|w\|_{L^{\infty}_T(\cX^{0}_{a,\sigma})}^{1-\frac{1}{2\gamma}} \|w\|_{L^{1}_T(\cX^{2\gamma}_{a,\sigma})}^{\frac{1}{2\gamma}}\\
&\quad +CT^{1-\frac{1}{2\gamma}}\|w\|_{L^{\infty}_T(\cX^{0}_{a,\sigma})}\|v\|_{L^{\infty}_T(\cX^{0}_{a,\sigma})}^{1-\frac{1}{2\gamma}} \|v\|_{L^{1}_T(\cX^{2\gamma}_{a,\sigma})}^{\frac{1}{2\gamma}},
\end{align}
for $\gamma\geq \frac{1}{2}$. In particular, it holds
\begin{align} \label{eq:A2}
	 \Vert B(u,u) \Vert_{L^{\infty}_T(\cX^{0}_{a,\sigma})} \leq  CT^{1-\frac{1}{2\gamma}}\|u\|_{L^{\infty}_T(\cX^{0}_{a,\sigma})}^{2-\frac{1}{2\gamma}}\|u\|_{L^{1}_T(\cX^{2\gamma}_{a,\sigma})}^{\frac{1}{2\gamma}}\leq  C_\gamma T^{1-\frac{1}{2\gamma}} R^{\frac{1}{2\gamma}}\|u_0\|_{\cX^{0}_{a,\sigma}}^{2-\frac{1}{2\gamma}},
\end{align}
for $\gamma \geq\frac{1}{2}$ and for all $u\in \mathcal{F}_T$. Analogously, one deduces
\begin{align*} 
	\nonumber \Vert B(w,v) \Vert_{L^1_T(\cX^{2\gamma}_{a,\sigma})} 
\nonumber&\leq\int_0^T \int_{\mathbb{R}^3} |\xi|^{2\gamma}e^{a|\xi|^{\frac{1}{\sigma}}}\int_0^te^{-(t-\tau)|\xi|^{2\gamma}}|\mathcal{F}[\mathbb{P}(w\cdot \nabla v)](\xi,\tau)|\,d\tau d\xi dt\\
\nonumber&\leq\int_{\mathbb{R}^3} |\xi|^{2\gamma+1}e^{a|\xi|^{\frac{1}{\sigma}}}\int_0^T\int_0^te^{-(t-\tau)|\xi|^{2\gamma}}|\mathcal{F}[(v\otimes w)](\xi,\tau)|\,d\tau  dt d\xi\\
\nonumber&=\int_{\mathbb{R}^3} |\xi|^{2\gamma+1}e^{a|\xi|^{\frac{1}{\sigma}}}\int_0^T |\mathcal{F}[(v\otimes w)](\xi,\tau)|\int_\tau^Te^{-(t-\tau)|\xi|^{2\gamma}}\,dt  d\tau d\xi\\
\nonumber&\leq\int_0^T \|(v\otimes w)(\tau)\|_{\cX^{1}_{a,\sigma}} d\tau.
\end{align*}
By (\ref{w18}) and (\ref{w17}), we obtain
\begin{align} \label{w19}
	\nonumber\Vert B(w,v) \Vert_{L^1_T(\cX^{2\gamma}_{a,\sigma})}&\leq  CT^{1-\frac{1}{2\gamma}}\|v\|_{L^{\infty}_T(\cX^{0}_{a,\sigma})}\|w\|_{L^{\infty}_T(\cX^{0}_{a,\sigma})}^{1-\frac{1}{2\gamma}} \|w\|_{L^{1}_T(\cX^{2\gamma}_{a,\sigma})}^{\frac{1}{2\gamma}}\\
&\quad +CT^{1-\frac{1}{2\gamma}}\|w\|_{L^{\infty}_T(\cX^{0}_{a,\sigma})}\|v\|_{L^{\infty}_T(\cX^{0}_{a,\sigma})}^{1-\frac{1}{2\gamma}} \|v\|_{L^{1}_T(\cX^{2\gamma}_{a,\sigma})}^{\frac{1}{2\gamma}},
\end{align}
for $\gamma\geq\frac{1}{2}$. In particular, we can write the following inequality:
\begin{align}\label{w20}
	 \Vert B(u,u) \Vert_{L^1_T(\cX^{2\gamma}_{a,\sigma})} \leq  CT^{1-\frac{1}{2\gamma}}\|u\|_{L^{\infty}_T(\cX^{0}_{a,\sigma})}^{2-\frac{1}{2\gamma}}\|u\|_{L^{1}_T(\cX^{2\gamma}_{a,\sigma})}^{\frac{1}{2\gamma}}\leq  C_\gamma T^{1-\frac{1}{2\gamma}} R^{\frac{1}{2\gamma}}\|u_0\|_{\cX^{0}_{a,\sigma}}^{2-\frac{1}{2\gamma}},
\end{align}
for $\gamma \geq\frac{1}{2}$ and for all $u\in \mathcal{F}_T$.

By (\ref{w53}), it is also true that
\begin{align}\label{w181} 
	\nonumber \Vert B'(\phi,\psi,\theta) (t)\Vert_{\cX^{0}_{a,\sigma}} &\leq \int_0^t \|e^{-(t-\tau)(-\Delta)^{\gamma}}\mathbb{P}[(\phi\cdot \psi)\theta](\tau)\|_{\cX^{0}_{a,\sigma}}\,d\tau\\
\nonumber&=\int_0^t \int_{\mathbb{R}^3}e^{a|\xi|^{\frac{1}{\sigma}}}e^{-(t-\tau)|\xi|^{2\gamma}}|\mathcal{F}\{\mathbb{P}[(\phi\cdot \psi)\theta]\}(\xi,\tau)|\,d\xi d\tau\\
\nonumber&\leq\int_0^t \int_{\mathbb{R}^3}e^{a|\xi|^{\frac{1}{\sigma}}}|\mathcal{F}[(\phi\cdot \psi)\theta](\xi,\tau)|\,d\xi d\tau\\
&=\int_0^t \|[(\phi\cdot \psi)\theta](\tau)\|_{\cX^{0}_{a,\sigma}} d\tau,
\end{align}
for all $t\in[0,T]$. Then, from the inequality  (\ref{lem:produto}), it follows that
\begin{align*} 
	\nonumber \Vert B'(\phi,\psi,\theta) (t)\Vert_{\cX^{0}_{a,\sigma}}&\leq  C\int_0^T \|\phi\|_{\cX^{0}_{a,\sigma}}\| \psi\|_{\cX^{0}_{a,\sigma}}\|\theta\|_{\cX^{0}_{a,\sigma}}d\tau\\
&\leq  CT\|\phi\|_{L^{\infty}_T(\cX^{0}_{a,\sigma})}\|\psi\|_{L^{\infty}_T(\cX^{0}_{a,\sigma})}\|\theta\|_{L^{\infty}_T(\cX^{0}_{a,\sigma})},
\end{align*}
 for all $t\in[0,T]$. Thereby, one reaches
\begin{align}\label{w171} 
	 \Vert B'(\phi,\psi,\theta) \Vert_{L^\infty_T(\cX^{0}_{a,\sigma})} &\leq  CT\|\phi\|_{L^{\infty}_T(\cX^{0}_{a,\sigma})}\|\psi\|_{L^{\infty}_T(\cX^{0}_{a,\sigma})}\|\theta\|_{L^{\infty}_T(\cX^{0}_{a,\sigma})}.
\end{align}
In particular, it holds
\begin{align} \label{eq:A21}
	 \Vert B'(u,u,u) \Vert_{L^{\infty}_T(\cX^{0}_{a,\sigma})} \leq  CT\|u\|_{L^{\infty}_T(\cX^{0}_{a,\sigma})}^3\leq   8C T \|u_0\|_{\cX^{0}_{a,\sigma}}^3,
\end{align}
for all $u\in \mathcal{F}_T$. Similarly, we can write
\begin{align*} 
	\nonumber \Vert B'(\phi,\psi,\theta) \Vert_{L^1_T(\cX^{2\gamma}_{a,\sigma})} 
\nonumber&\leq\int_0^T \int_{\mathbb{R}^3} |\xi|^{2\gamma}e^{a|\xi|^{\frac{1}{\sigma}}}\int_0^te^{-(t-\tau)|\xi|^{2\gamma}}|\mathcal{F}\{\mathbb{P}[(\phi\cdot \psi)\theta]\}(\xi,\tau)|\,d\tau d\xi dt\\
\nonumber&=\int_{\mathbb{R}^3} |\xi|^{2\gamma}e^{a|\xi|^{\frac{1}{\sigma}}}\int_0^T\int_0^te^{-(t-\tau)|\xi|^{2\gamma}}|\mathcal{F}[(\phi\cdot \psi)\theta](\xi,\tau)|\,d\tau  dt d\xi\\
\nonumber&=\int_{\mathbb{R}^3} |\xi|^{2\gamma}e^{a|\xi|^{\frac{1}{\sigma}}}\int_0^T |\mathcal{F}[(\phi\cdot \psi)\theta](\xi,\tau)|\int_\tau^Te^{-(t-\tau)|\xi|^{2\gamma}}\,dt  d\tau d\xi\\
\nonumber&\leq\int_0^T \|(\phi\cdot \psi)\theta(\tau)\|_{\cX^{0}_{a,\sigma}} d\tau.
\end{align*}
By applying (\ref{w181}) and (\ref{w171}), one has
\begin{align} \label{w191}
	\Vert B'(\phi,\psi,\theta) \Vert_{L^1_T(\cX^{2\gamma}_{a,\sigma})}&\leq  CT\|\phi\|_{L^{\infty}_T(\cX^{0}_{a,\sigma})}\|\psi\|_{L^{\infty}_T(\cX^{0}_{a,\sigma})}\|\theta\|_{L^{\infty}_T(\cX^{0}_{a,\sigma})}.
\end{align}
Particularly, we obtain
\begin{align}\label{w201}
	 \Vert B'(u,u,u) \Vert_{L^1_T(\cX^{2\gamma}_{a,\sigma})}\leq CT\|u\|_{L^{\infty}_T(\cX^{0}_{a,\sigma})}^{3}\leq   8C T \|u_0\|_{\cX^{0}_{a,\sigma}}^3,
\end{align}
for all $u\in \mathcal{F}_T$.

%{\color{red}
%To keep the constants compatible with the later thresholds, fix once and for all constants
%$K_2,K_3>0$ which dominate the preceding bilinear and trilinear bounds.  With
%$\vartheta=1-\frac1{2\gamma}$, the estimates can be summarized as
%\begin{align}
%\Vert B(u,v)\Vert_T&\le K_2T^{\vartheta}\Vert u\Vert_T\Vert v\Vert_T,\label{fixedK2}\\
%\Vert B'(f,g,h)\Vert_T&\le K_3T\Vert f\Vert_T\Vert g\Vert_T\Vert h\Vert_T.\label{fixedK3}
%\end{align}
%The constants include all Fourier-normalization, dimensional, and numerical factors, including
%the factor $8$ produced by $\Vert u\Vert_{L^\infty_T\cX^0}\le2\Vert u_0\Vert_{\cX^0}$.
%The same constants are used in every choice of smallness radius and lifespan below.
%Strong continuity of the Duhamel terms in $\cX^0_{a,\sigma}$ follows by dominated convergence
%and the standard decomposition of the difference of two time convolutions; hence $\Psi$ really
%takes values in the continuous component of $\mathcal X_T$.
%}

Thus, combining (\ref{w24}), (\ref{eq:A1}), (\ref{eq:A2}) and (\ref{eq:A21}), we have
\begin{align}\label{w21}
\nonumber	\Vert \Psi[u] \Vert_{L^{\infty}_T(\cX^{0}_{a,\sigma})} &\leq \Vert e^{-t(-\Delta)^{\gamma}} u_0 \Vert_{L^{\infty}_T(\cX^{0}_{a,\sigma})} + \Vert B(u,u) \Vert_{L^{\infty}_T(\cX^{0}_{a,\sigma})} + \alpha\Vert B'(u,u,u) \Vert_{L^{\infty}_T(\cX^{0}_{a,\sigma})}\\
&\leq \Vert u_0 \Vert_{\cX^{0}_{a,\sigma}}+ C_\gamma [T^{1-\frac{1}{2\gamma}} R^{\frac{1}{2\gamma}}\|u_0\|_{\cX^{0}_{a,\sigma}}^{2-\frac{1}{2\gamma}}+\alpha  T \|u_0\|_{\cX^{0}_{a,\sigma}}^3],
\end{align}
for all $u\in \mathcal{F}_T$. Moreover, by (\ref{w24}), (\ref{w16}), (\ref{w20}) and (\ref{w201}), one concludes
\begin{align} \label{w22}
	\nonumber\Vert \Psi[u] \Vert_{L^{1}_T(\cX^{2\gamma}_{a,\sigma})} &\leq \Vert e^{-t(-\Delta)^{\gamma}} u_0 \Vert_{L^{1}_T(\cX^{2\gamma}_{a,\sigma})} + \Vert B(u,u) \Vert_{L^{1}_T(\cX^{2\gamma}_{a,\sigma})} + \alpha\Vert B'(u,u,u) \Vert_{L^{1}_T(\cX^{2\gamma}_{a,\sigma})}\\
&\leq \Vert u_0 \Vert_{\cX^{0}_{a,\sigma}}+C_\gamma [T^{1-\frac{1}{2\gamma}} R^{\frac{1}{2\gamma}}\|u_0\|_{\cX^{0}_{a,\sigma}}^{2-\frac{1}{2\gamma}}+\alpha  T \|u_0\|_{\cX^{0}_{a,\sigma}}^3],
\end{align}
for all $u\in \mathcal{F}_T$. Now, let us prove all the items of Theorem \ref{thm:solucaolocal}.

\bigskip
\noindent \underline{Proof of Theorem \ref{thm:solucaolocal} i)}: Assume that $\gamma=\frac{1}{2}$ and $\alpha=0$.
\\\\
%\bigskip
%\noindent \emph{Case 1}: Consider $\alpha=0$.
%\bigskip
If we choose $R := \frac{1}{8C}$ (with $C = C_{\frac{1}{2}}$ given in (\ref{w21})), (\ref{w21}) becomes
\begin{align} \label{w23}
	\Vert \Psi[u] \Vert_{L^{\infty}_T(\cX^{0}_{a,\sigma})} \leq 2 \Vert u_0 \Vert_{\cX^{0}_{a,\sigma}},\quad\forall u\in\mathcal{F}_T.
\end{align}
Then, by supposing 
\begin{align}\label{dadoinicial}
\Vert u_0 \Vert_{\cX^{0}_{a,\sigma}} < \frac{R}{2}=\frac{1}{16C}=:C',
\end{align}
 it follows, from (\ref{w22}) and our choice for $R$, that
\begin{equation} \label{eq:A6}
	\Vert \Psi[u] \Vert_{L^1_T(\cX^{1}_{a,\sigma})} \leq R.
\end{equation}
Therefore, by (\ref{w23}) and (\ref{eq:A6}), we have $\Psi[u] \in \mathcal{F}_{T}$, for all $u\in \mathcal{F}_T$. This means that $\Psi(\mathcal{F}_{T}) \subseteq \mathcal{F}_{T}$.

Now, we shall prove the inequality (\ref{eq:claim2}). In fact, by using the equality 
$$
	u \cdot \nabla u - v \cdot \nabla v = u \cdot \nabla (u - v) + (u- v) \cdot \nabla v,
$$
one can write
\begin{align}\label{w38}
	B(u,u) - B(v,v) = B(u,u-v) + B(u-v,v).
\end{align}
Consequently, by (\ref{w17}) and (\ref{w19}), one deduces
	\begin{align}\label{w32}
\nonumber		\Vert B(u,u) - B(v,v) \Vert_{T} &\leq \Vert B(u,u-v) \Vert_{T} + \Vert B(u-v,v) \Vert_T\\
\nonumber&\leq C[\|u-v\|_{L^{\infty}_T(\cX^{0}_{a,\sigma})}\|u\|_{L^{1}_T(\cX^{1}_{a,\sigma})}+\|u-v\|_{L^{1}_T(\cX^{1}_{a,\sigma})}\|u\|_{L^{\infty}_T(\cX^{0}_{a,\sigma})}]\\
\nonumber&\quad+ C[\|u-v\|_{L^{\infty}_T(\cX^{0}_{a,\sigma})}\|v\|_{L^{1}_T(\cX^{1}_{a,\sigma})}+\|u-v\|_{L^{1}_T(\cX^{1}_{a,\sigma})}\|v\|_{L^{\infty}_T(\cX^{0}_{a,\sigma})}]\\
		% &\leq 2c \left( \Vert u \Vert_{L^{\infty}_T(\cX^{0}_{a,\sigma})} \Vert u - v \Vert_{L^1_T(\cX^{1}_{a,\sigma})} + \Vert u - v \Vert_{L^{\infty}_T(\cX^{0}_{a,\sigma})} \Vert u \Vert_{L^1_T(\cX^{1}_{a,\sigma})} \right)\\ 
		% &+ 2c \left( \Vert u - v \Vert_{L^{\infty}_T(\cX^{0}_{a,\sigma})} \Vert v \Vert_{L^1_T(\cX^{1}_{a,\sigma})} + \Vert v \Vert_{L^{\infty}_T(\cX^{0}_{a,\sigma})} \Vert u - v \Vert_{L^1_T(\cX^{1}_{a,\sigma})} \right)
		&\leq4C [ \Vert u_0 \Vert_{\cX^{0}_{a,\sigma}} + \tfrac{R}{2}] \Vert u - v \Vert_T,
	\end{align}
for all $u,v\in \mathcal{F}_T$. Therefore, by (\ref{w24}), (\ref{dadoinicial}) and the fact that $R=\frac{1}{8C}$, one infers
\begin{align*}
	\Vert \Psi[u] - \Psi[v] \Vert_{T} &\leq \Vert B(u,u) - B(v,v) \Vert_{T} \leq 4C [ \Vert u_0 \Vert_{\cX^{0}_{a,\sigma}} + \tfrac{R}{2}] \Vert u - v \Vert_T\\
&\leq 4C [\tfrac{R}{2} + \tfrac{R}{2}]\Vert u - v \Vert_T= \frac{1}{2}  \Vert u - v \Vert_T,
\end{align*}
for all $u,v\in \mathcal{F}_T$. This establishes the proof of (\ref{eq:claim2}).

%\subsection{Final steps of the proof}

As a result, by (\ref{eq:claim1}) and (\ref{eq:claim2}), we can apply  Banach Fixed Point Theorem to the application $\Psi$ in order to obtain a unique mild solution $u \in \mathcal{F}_T\subseteq \mathcal{X}_T$ for the Navier-Stokes equations (\ref{NS}) (see (\ref{w24})). Moreover, by noticing that $u \in \mathcal{F}_T$, we conclude that
$$\Vert u(t) \Vert_{\cX^0_{a,\sigma}}\leq \Vert u \Vert_{L^{\infty}_T(\cX^0_{a,\sigma})} \leq 2 \Vert u_0 \Vert_{\cX^{0}_{a,\sigma}},\quad\forall t\in[0,T],$$
for any $T>0$. The fixed points obtained for two horizons $T_1<T_2$ agree on $[0,T_1]$ by
uniqueness on sufficiently short subintervals and continuation over overlapping intervals.
Consequently they define one global solution, rather than unrelated solutions for each $T$.

In addition, if we consider that 
  \begin{align}\label{w27}
  u \in C([0,T^*), \cX^0_{a,\sigma}(\bR^3)) \cap L^1_{\mathrm{loc}}([0,T^*), \cX^1_{a,\sigma}(\bR^3)),
  \end{align}
  with $\|u_0\|_{\cX^0_{a,\sigma}}<C'$, is the  maximal solution for the Navier-Stokes equations (\ref{NS}); then, $T^*=\infty$. In fact, suppose, by absurdity, that $T^*<\infty$. Thus, the arguments proved above imply that there is a unique solution for the Navier-Stokes equations (\ref{NS}) 
  \begin{align*}
  u \in C_{T^*+T}(\cX^0_{a,\sigma}(\bR^3)) \cap L^1_{T^*+T}( \cX^1_{a,\sigma}(\bR^3))
  \end{align*}
  for any $T>0$. However, this contradicts the maximality of $T^*$. Therefore, by (\ref{w27}), one concludes
  \begin{align*}
  u \in C([0,\infty), \cX^0_{a,\sigma}(\bR^3)) \cap L^1([0,\infty), \cX^1_{a,\sigma}(\bR^3)).
  \end{align*}
  
\caixa

\noindent \underline{Proof of Theorem \ref{thm:solucaolocal} ii)}: Assume that $\gamma=\frac{1}{2}$ and $\alpha>0$.
\\\\
%\bigskip
%\noindent \emph{Case 2}: Consider $\alpha>0$.
%\bigskip
Now, choose $R := \frac{3}{40C}$ and
$T:=\bar T=[48\alpha C\Vert u_0\Vert_{\cX^0_{a,\sigma}}^2]^{-1}$ (with $C = C_\frac{1}{2}$
 and $T$ given in (\ref{w21})), to reach, from (\ref{w21}), the following:
\begin{align} \label{w2321}
	\Vert \Psi[u] \Vert_{L^{\infty}_{\bar{T}}(\cX^{0}_{a,\sigma})} \leq 2 \Vert u_0 \Vert_{\cX^{0}_{a,\sigma}},\quad\forall u\in\mathcal{F}_{\bar T}.
\end{align}
Thus, by assuming
\begin{align}\label{dadoinicial2}
\Vert u_0 \Vert_{\cX^{0}_{a,\sigma}} < \frac{R}{3}=\frac{1}{40 C}=:C',
\end{align}
 we obtain, from (\ref{w22}) and our choice for $R$ and $T=\bar T$, that
\begin{equation} \label{eq:A621}
	\Vert \Psi[u] \Vert_{L^1_{\bar{T}}(\cX^{1}_{a,\sigma})} \leq R.
\end{equation}
As a result, by (\ref{w2321}) and (\ref{eq:A621}), one deduces that $\Psi[u] \in \mathcal{F}_{{\bar{T}}}$, for all $u\in \mathcal{F}_{\bar T}$. Therefore, $\Psi(\mathcal{F}_{{\bar{T}}}) \subseteq \mathcal{F}_{{\bar{T}}}$.

Let us show the inequality (\ref{eq:claim2}) in this case as well. In fact, it is true that
$$
	|u|^2u - |v|^2 v = |u|^2 (u - v) + [(u- v) \cdot (u+v) ] v
$$
leads us to infer
$$
	B'(u,u,u) - B'(v,v,v) = B'(u,u,u-v) + B'(u-v,u+v,v).
$$
As a consequence, by (\ref{w171}) and (\ref{w191}), we have
	\begin{align}\label{w33}
\nonumber		\Vert B'(u,u,u) - B'(v,v,v) \Vert_{\bar{T}} &\leq \Vert B'(u,u,u-v) \Vert_{\bar{T}} + \Vert B'(u-v,u+v,v) \Vert_{\bar{T}}\\
\nonumber&\leq C\bar T[\|u\|_{L^{\infty}_{\bar T}(\cX^{0}_{a,\sigma})}^2\|u-v\|_{\bar{T}}+\|u-v\|_{\bar{T}}\|u+v\|_{L^{\infty}_{\bar{T}}(\cX^{0}_{a,\sigma})}\|v\|_{L^{\infty}_{\bar{T}}(\cX^{0}_{a,\sigma})}]\\
		% &\leq 2c \left( \Vert u \Vert_{L^{\infty}_T(\cX^{0}_{a,\sigma})} \Vert u - v \Vert_{L^1_T(\cX^{1}_{a,\sigma})} + \Vert u - v \Vert_{L^{\infty}_T(\cX^{0}_{a,\sigma})} \Vert u \Vert_{L^1_T(\cX^{1}_{a,\sigma})} \right)\\ 
		% &+ 2c \left( \Vert u - v \Vert_{L^{\infty}_T(\cX^{0}_{a,\sigma})} \Vert v \Vert_{L^1_T(\cX^{1}_{a,\sigma})} + \Vert v \Vert_{L^{\infty}_T(\cX^{0}_{a,\sigma})} \Vert u - v \Vert_{L^1_T(\cX^{1}_{a,\sigma})} \right)
		&\leq 12 C {\bar{T}}  \Vert u_0 \Vert_{\cX^{0}_{a,\sigma}}^2\Vert u - v \Vert_{\bar{T}},
	\end{align}
for all $u,v\in \mathcal{F}_{\bar{T}}$. Thereby, by (\ref{w24}), (\ref{w32}), (\ref{w33}), (\ref{dadoinicial2}) and our choice for $R$ and $T=\bar T$, it follows that
\begin{align*}%\label{w34}
	\nonumber\Vert \Psi[u] - \Psi[v] \Vert_{\bar{T}} &\leq\Vert B(u,u) - B(v,v) \Vert_{\bar{T}} +\alpha\Vert B'(u,u,u) - B'(v,v,v) \Vert_{\bar{T}}\\
\nonumber &\leq 4C\bigl( \Vert u_0 \Vert_{\cX^{0}_{a,\sigma}} + \tfrac{R}{2}+3\alpha \bar{T}\Vert u_0 \Vert_{\cX^{0}_{a,\sigma}}^2\bigr) \Vert u - v \Vert_{\bar{T}}\\
\nonumber&\leq 4C\left(\frac{5R}{6}+\frac{1}{16C}\right)\Vert u-v\Vert_{\bar T}=\frac12\Vert u-v\Vert_{\bar T},
\end{align*}
for all $u,v\in \mathcal{F}_{\bar T}$. This establishes the proof of (\ref{eq:claim2}).

%\subsection{Final steps of the proof}

Therefore, by (\ref{eq:claim1}) and (\ref{eq:claim2}), we can apply  Banach Fixed Point Theorem to the application $\Psi$ in order to obtain a unique mild solution $u \in \mathcal{F}_{\bar{T}}\subseteq \mathcal{X}_{\bar{T}}$ for the Navier-Stokes equations (\ref{NS}) (see (\ref{w24})). Moreover, by noticing that $u \in \mathcal{F}_{\bar{T}}$, we conclude that
$$\Vert u(t) \Vert_{\cX^0_{a,\sigma}}\leq \Vert u \Vert_{L^{\infty}_{\bar{T}}(\cX^0_{a,\sigma})} \leq 2 \Vert u_0 \Vert_{\cX^{0}_{a,\sigma}},\quad\forall t\in[0,{\bar{T}}],$$
for some $\bar T = \bar T (\alpha,u_0)>0.$

\caixa

\noindent \underline{Proof of Theorem \ref{thm:solucaolocal} iii)}: Assume that $\gamma>\frac{1}{2}$ and $\alpha\geq0$.
\\\\
%\bigskip
%\noindent \emph{Case 2}: Consider $\alpha>0$.
%\bigskip
First of all, by taking into account (\ref{w35}),  notice that there is  $T_1>0$ such that
\begin{align}\label{w363}
\Vert e^{-t(-\Delta)^\gamma} u_0 \Vert_{L^{1}_T(\cX^{2\gamma}_{a,\sigma})}<\frac{R}{3},\quad\forall T\in[0,T_1],
\end{align}
since $u_0\in \cX^{0}_{a,\sigma} (\mathbb{R}^3)$ (apply dominated convergence theorem).
Put $A=\Vert u_0\Vert_{\cX^0_{a,\sigma}}$,
$\lambda=(2\gamma)^{-1}$ and $\vartheta=1-\lambda>0$. % If $A=0$, the zero solution is immediate.  If $A>0$, 
Fix $R>0$ and choose $0<T:=\bar T\le T_1$ so small that
\begin{align}\label{w37}
C_\gamma\bar T^{\vartheta}R^\lambda A^{2-\lambda}
&\le \min\Big\{\frac{A}{2},\frac{R}{6}\Big\},\nonumber\\
\alpha C_\gamma\bar T A^3&\le \min\Big\{\frac{A}{2},\frac{R}{6}\Big\},\nonumber\\
C_\gamma[\bar T^{\vartheta}\bigl(A+A^{1-\lambda}R^\lambda\bigr)
+\alpha \bar T A^2]&\le\frac12,
\end{align}
where  $C_\gamma$ is given in (\ref{w21}) or (\ref{w223}) or (\ref{w343}). The restrictions containing $\alpha$ are omitted when $\alpha=0$.  Such a time exists because
$\vartheta>0$ and all constants above are fixed before $\bar T$ is chosen.
Then (\ref{w21}) and (\ref{w37}) give
\begin{align} \label{w232}
	\Vert \Psi[u] \Vert_{L^{\infty}_{\bar{T}}(\cX^{0}_{a,\sigma})} \leq 2 \Vert u_0 \Vert_{\cX^{0}_{a,\sigma}},\quad\forall u\in\mathcal{F}_{\bar T}.
\end{align}
%Thus, combining (\ref{w36}), (\ref{eq:A2}) and (\ref{eq:A21}), we have
%\begin{align}\label{w213}
%\nonumber	\Vert \Psi[u] \Vert_{L^{\infty}_T(\cX^{0}_{a,\sigma})} &\leq \Vert e^{-t(-\Delta)^{\frac{1}{2}}} u_0 \Vert_{L^{\infty}_T(\cX^{0}_{a,\sigma})} + \Vert B(u,u) %\Vert_{L^{\infty}_T(\cX^{0}_{a,\sigma})} + \alpha\Vert B'(u,u,u) \Vert_{L^{\infty}_T(\cX^{0}_{a,\sigma})}\\
%&\leq \frac{R}{3}+ C_\gamma [T^{1-\frac{1}{2\gamma}} R^{\frac{1}{2\gamma}}\|u_0\|_{\cX^{0}_{a,\sigma}}^{2-\frac{1}{2\gamma}}+\alpha  T \|u_0\|_{\cX^{0}_{a,\sigma}}^3].
%\end{align}
Furthermore, by (\ref{w363}), (\ref{w20}) and (\ref{w201}), one concludes
\begin{align} \label{w223}
	\nonumber\Vert \Psi[u] \Vert_{L^{1}_T(\cX^{2\gamma}_{a,\sigma})} &\leq \Vert e^{-t(-\Delta)^{\gamma}} u_0 \Vert_{L^{1}_T(\cX^{2\gamma}_{a,\sigma})} + \Vert B(u,u) \Vert_{L^{1}_T(\cX^{2\gamma}_{a,\sigma})} + \alpha\Vert B'(u,u,u) \Vert_{L^{1}_T(\cX^{2\gamma}_{a,\sigma})}\\
&\leq \frac{R}{3}+C_\gamma [T^{1-\frac{1}{2\gamma}} R^{\frac{1}{2\gamma}}\|u_0\|_{\cX^{0}_{a,\sigma}}^{2-\frac{1}{2\gamma}}+\alpha  T \|u_0\|_{\cX^{0}_{a,\sigma}}^3],
\end{align}
for all $T\in [0,T_1]$. Consequently, we obtain, from   (\ref{w37}) and (\ref{w223}), that
\begin{equation} \label{eq:A62}
	\Vert \Psi[u] \Vert_{L^1_{\bar{T}}(\cX^{2\gamma}_{a,\sigma})} \leq R.
\end{equation}
Therefore, by (\ref{w232}) and (\ref{eq:A62}), one obtains $\Psi[u] \in \mathcal{F}_{{\bar{T}}}$, for all $u\in \mathcal{F}_{\bar T}$. Therefore, $\Psi(\mathcal{F}_{{\bar{T}}}) \subseteq \mathcal{F}_{{\bar{T}}}$.

In addition, by using (\ref{w38}), (\ref{w17}) and (\ref{w19}), we reach
	\begin{align}\label{w323}
\nonumber		\Vert B(u,u) - B(v,v) \Vert_{T} &\leq \Vert B(u,u-v) \Vert_{T} + \Vert B(u-v,v) \Vert_T\\
\nonumber&\leq CT^{1-\frac{1}{2\gamma}}\|u-v\|_{L^{\infty}_T(\cX^{0}_{a,\sigma})}\|u\|_{L^{\infty}_T(\cX^{0}_{a,\sigma})}^{1-\frac{1}{2\gamma}}\|u\|_{L^{1}_T(\cX^{2\gamma}_{a,\sigma})}^{\frac{1}{2\gamma}}
\\
\nonumber&\quad+CT^{1-\frac{1}{2\gamma}}\|u\|_{L^{\infty}_T(\cX^{0}_{a,\sigma})}\|u-v\|_{L^{\infty}_T(\cX^{0}_{a,\sigma})}^{1-\frac{1}{2\gamma}}\|u-v\|_{L^{1}_T(\cX^{2\gamma}_{a,\sigma})}^{\frac{1}{2\gamma}}\\
\nonumber&\quad+ CT^{1-\frac{1}{2\gamma}}\|v\|_{L^{\infty}_T(\cX^{0}_{a,\sigma})}\|u-v\|_{L^{\infty}_T(\cX^{0}_{a,\sigma})}^{1-\frac{1}{2\gamma}}\|u-v\|_{L^{1}_T(\cX^{2\gamma}_{a,\sigma})}^{\frac{1}{2\gamma}}
\\
\nonumber&\quad+CT^{1-\frac{1}{2\gamma}}\|u-v\|_{L^{\infty}_T(\cX^{0}_{a,\sigma})}\|v\|_{L^{\infty}_T(\cX^{0}_{a,\sigma})}^{1-\frac{1}{2\gamma}}\|v\|_{L^{1}_T(\cX^{2\gamma}_{a,\sigma})}^{\frac{1}{2\gamma}}\\
				&\leq C_\gamma T^{1-\frac{1}{2\gamma}} [ \Vert u_0 \Vert_{\cX^{0}_{a,\sigma}} + \Vert u_0 \Vert_{\cX^{0}_{a,\sigma}}^{1-\frac{1}{2\gamma}}R^{\frac{1}{2\gamma}}] \Vert u - v \Vert_T,
	\end{align}
for all $u,v\in \mathcal{F}_T$.

Now, we are ready to prove the inequality (\ref{eq:claim2}) in this last case. In fact, 
%$$
%	|u|^2u - |v|^2 v = |u|^2 (u - v) + [(u- v) \cdot (u+v) ] v
%$$
%leads us to infer
%$$
%	B'(u,u,u) - B'(v,v,v) = B'(u,u,u-v) + B'(u-v,u+v,v).
%$$
%As a consequence, by (\ref{w171}) and (\ref{w191}), we have
%	\begin{align}\label{w33}
%\nonumber		\Vert B'(u,u,u) - B'(v,v,v) \Vert_{\bar{T}} &\leq \Vert B'(u,u,u-v) \Vert_{\bar{T}} + \Vert B'(u-v,u+v,v) \Vert_{\bar{T}}\\
%\nonumber&\leq C\bar %T[\|u\|_{L^{\infty}_T(\cX^{0}_{a,\sigma})}^2\|u-v\|_{\bar{T}}+\|u-v\|_{\bar{T}}\|u+v\|_{L^{\infty}_{\bar{T}}(\cX^{0}_{a,\sigma})}\|v\|_{L^{\infty}_{\bar{T}}(\cX^{0}_{a,\sigma})}]\\
		% &\leq 2c \left( \Vert u \Vert_{L^{\infty}_T(\cX^{0}_{a,\sigma})} \Vert u - v \Vert_{L^1_T(\cX^{1}_{a,\sigma})} + \Vert u - v \Vert_{L^{\infty}_T(\cX^{0}_{a,\sigma})} \Vert u %\Vert_{L^1_T(\cX^{1}_{a,\sigma})} \right)\\ 
		% &+ 2c \left( \Vert u - v \Vert_{L^{\infty}_T(\cX^{0}_{a,\sigma})} \Vert v \Vert_{L^1_T(\cX^{1}_{a,\sigma})} + \Vert v \Vert_{L^{\infty}_T(\cX^{0}_{a,\sigma})} \Vert u - v %\Vert_{L^1_T(\cX^{1}_{a,\sigma})} \right)
%		&\leq12C {\bar{T}}  \Vert u_0 \Vert_{\cX^{0}_{a,\sigma}}^2\Vert u - v \Vert_{\bar{T}},
%	\end{align}
%for all $u,v\in \mathcal{F}_{\bar{T}}$. Thereby, 
by (\ref{w33}), (\ref{w323}) and (\ref{w37}), it follows that
\begin{align}\label{w343}
	\nonumber\Vert \Psi[u] - \Psi[v] \Vert_{\bar{T}} &\leq \Vert B(u,u) - B(v,v) \Vert_{\bar{T}} +\alpha\Vert B'(u,u,u) - B'(v,v,v) \Vert_{\bar{T}}\\
\nonumber &\leq C_\gamma\Bigl[\bar T^{\vartheta}\bigl(A+A^{1-\lambda}R^\lambda\bigr)+\alpha \bar T A^2\Bigr]\Vert u-v\Vert_{\bar T}\\
&\leq \frac{1}{2}  \Vert u - v \Vert_{\bar{T}},
\end{align}
for all $u,v\in \mathcal{F}_{\bar T}$. This establishes the proof of (\ref{eq:claim2}).

%\subsection{Final steps of the proof}

Hence, by (\ref{eq:claim1}) and (\ref{eq:claim2}), by applying  Banach Fixed Point Theorem to the application $\Psi$, one infers that there exists a unique mild solution $u \in \mathcal{F}_{\bar{T}}\subseteq \mathcal{X}_{\bar{T}}$ for the Navier-Stokes equations (\ref{NS}) (see (\ref{w24})). Furthermore, by noticing that $u \in \mathcal{F}_{\bar{T}}$, one deduces
$$\Vert u(t) \Vert_{\cX^0_{a,\sigma}}\leq \Vert u \Vert_{L^{\infty}_{\bar{T}}(\cX^0_{a,\sigma})} \leq 2 \Vert u_0 \Vert_{\cX^{0}_{a,\sigma}},\quad\forall t\in[0,{\bar{T}}],$$
for some $\bar T = \bar T (\alpha,\gamma,u_0)>0.$

\caixa

\bigskip
\noindent\textbf{Proof of Theorem \ref{thm:blowup}:}
%Let us present  the proof of  Theorem \ref{thm:blowup}.  
It is important to emphasize once again that we shall adapt and improve the proofs obtained by \cite{patricia,Bnovo,BNS,Robert} (see also references therein).

\bigskip
\noindent \underline{Proof of Theorem \ref{thm:blowup} i)}: 
Assume that $T^*<\infty$ and suppose, by absurdity, that $\displaystyle\limsup_{t \nearrow T^*} \Vert u(t) \Vert_{\cX^{0}_{a,\sigma}} < \infty$.
\\\\
As a result, by using the hypothesis that 
\begin{align*}
  u \in C([0,T^*), \cX^0_{a,\sigma}(\bR^3)) \cap L^1_{\mathrm{loc}}([0,T^*), \cX^{2\gamma}_{a,\sigma}(\bR^3))
\end{align*}
is the maximal solution for the Navier-Stokes equations (\ref{NS}), we can  write 
\begin{align}\label{w41}
	\Vert u(t) \Vert_{\cX^{0}_{a,\sigma}} \leq M_* ,\quad\forall t\in[0,T^*).
\end{align}

Apply a standard Fourier truncation (or Friedrichs approximation) to the mild
solution, perform the following scalar estimate for the smooth approximants, and then pass to
the limit by Fatou's lemma and dominated convergence.  The pressure contribution vanishes
because $\operatorname{div}u=0$.  In this justified sense, we infer
\begin{equation*}
	\frac{1}{2}\partial_t |\hat u|^2 + |\xi|^{2\gamma} |\hat u|^2 \leq |\hat u |\, | \widehat{u \cdot \nabla u}| + \alpha |\hat u |\, | \widehat{|u|^2u}|.
\end{equation*}
On the other hand, for every $\varepsilon > 0$, we have
\begin{align*}
	\partial_t \big[|\hat u(t)|^2 + \varepsilon \big]^{1/2} + \frac{|\xi|^{2\gamma} |\hat u(t)|^2}{[|\hat u(t)|^2 + \varepsilon]^{1/2}}  &\leq \frac{|\hat u(t)| |\widehat{u \cdot \nabla u}(t)|}{[|\hat u(t)|^2 + \varepsilon]^{1/2}} + \frac{\alpha |\hat u(t)||\widehat{|u|^2 u}(t)|}{[|\hat u(t)|^2 + \varepsilon]^{1/2}}\\
&\leq |\widehat{u \cdot \nabla u}(t)| + \alpha |\widehat{|u|^2 u}(t)|.
\end{align*}
Then, by integrating the inequality above over $[0,t]$ (with $t\in[0,T^*)$) and passing to the limit  $\varepsilon \searrow 0$, one concludes
$$
	|\hat{u}(t)| + \int_0^t |\xi|^{2\gamma} |\hat u(\tau)| \,d\tau \leq |\hat{u}_0| + \int_0^t |\widehat{u \cdot \nabla u}(\tau)| \,d\tau + \alpha \int_0^t |\widehat{|u|^2 u}(\tau)| \,d\tau,
$$
for all $t\in [0,T^*)$. Now, by multiplying   the inequality above by $e^{a|\xi|^{\frac{1}{\sigma}}}$ and integrating the result obtained over $\bR^3$, we have
\begin{align*}
\nonumber	\Vert u(t) \Vert_{\cX^{0}_{a,\sigma}} + \int_0^t \Vert u(\tau) \Vert_{\cX^{2\gamma}_{a,\sigma}} \,d\tau &\leq \Vert u_0 \Vert_{\cX^{0}_{a,\sigma}} + \int_0^t \Vert u \cdot \nabla u(\tau) \Vert_{\cX^{0}_{a,\sigma}} \,d\tau + \alpha \int_0^t \Vert |u|^2 u(\tau) \Vert_{\cX^{0}_{a,\sigma}}\,d\tau\\
&\leq \Vert u_0 \Vert_{\cX^{0}_{a,\sigma}} + \int_0^t \Vert u \otimes u(\tau) \Vert_{\cX^{1}_{a,\sigma}} \,d\tau + \alpha \int_0^t \Vert |u|^2 u(\tau) \Vert_{\cX^{0}_{a,\sigma}}\, d\tau,
\end{align*}
for all $t\in [0,T^*)$. Hence, by using  (\ref{lem:produto}), (\ref{lem:split}) and (\ref{w41}), we reach
\begin{align}\label{w50} 
	 \Vert u(t) \Vert_{\cX^{0}_{a,\sigma}} + \int_0^t \Vert u(\tau) \Vert_{\cX^{2\gamma}_{a,\sigma}} \,d\tau &\leq \Vert u_0 \Vert_{\cX^{0}_{a,\sigma}} + C\int_0^t \Vert u(\tau) \Vert_{\cX^{0}_{a,\sigma}} \Vert u(\tau) \Vert_{\cX^{1}_{a,\sigma}} \,d\tau + \alpha C \int_{0}^t \Vert u(\tau) \Vert_{\cX^{0}_{a,\sigma}}^3 \,d\tau\\
\nonumber &\leq \Vert u_0 \Vert_{\cX^{0}_{a,\sigma}} + CM_*^{2-\frac{1}{2\gamma}}\int_0^t  \Vert u(\tau) \Vert_{\cX^{2\gamma}_{a,\sigma}}^{\frac{1}{2\gamma}} \,d\tau + \alpha C M_*^3T^*,
\end{align}
for $\gamma\geq\frac{1}{2}$ and for all $t\in [0,T^*)$. As a result, by Young's inequality, it follows that
\begin{align*} 
	\nonumber \Vert u(t) \Vert_{\cX^{0}_{a,\sigma}} + \frac{1}{2}\int_0^t \Vert u(\tau) \Vert_{\cX^{2\gamma}_{a,\sigma}} \,d\tau
&\leq \Vert u_0 \Vert_{\cX^{0}_{a,\sigma}} + C_\gamma M_*^{\frac{4\gamma-1}{2\gamma-1}}T^* + \alpha C M_*^3T^*=:\frac{M^*}{2}<\infty,
\end{align*}
for $\gamma> \frac{1}{2}$ and for all $t\in [0,T^*)$ (recall that $u_0\in \cX^{0}_{a,\sigma}(\mathbb{R}^3)$). Thereby, we can write the following:
\begin{align} \label{w42}
	 \int_0^t \Vert u(\tau) \Vert_{\cX^{2\gamma}_{a,\sigma}} \,d\tau &\leq M^*, \quad\forall t\in[0,T^*).
\end{align}

On the other hand, since $u$ is a mild solution for the Navier-Stokes equations (\ref{NS}), it is true that
\begin{align*}
	u(t) = e^{-t(-\Delta)^{\gamma}} u_0 - \int_0^t e^{-(t-\tau)(-\Delta)^{\gamma}} \rP(u \cdot \nabla u)(\tau) \,d\tau-\alpha  \int_0^t e^{-(t-\tau)(-\Delta)^{\gamma}} \rP [|u|^2u](\tau) \,d\tau, 
\end{align*}
for all $t\in[0,T^*)$. Thus, we shall show that $(u(t_n))_{n\in \mathbb{N}}$ is a Cauchy sequence, where $(t_n)_{n\in\mathbb{N}} \subseteq [0,T^*)$ is a sequence  that satisfies $t_n \nearrow T^*$, as $n\rightarrow \infty$.  To this end, suppose that $t_n \geq t_k$ (the other case is analogous) to obtain
$$
	u(t_n) - u(t_k) = J_1 + J_2 + J_3 + \alpha J_4 + \alpha J_5,
$$
where
$$
	\begin{aligned}
		J_1 &= \big[e^{-t_n(-\Delta)^{\gamma}} - e^{-t_k(-\Delta)^{\gamma}}\big] u_0,\\
		J_2 &= \int_0^{t_k} \big[e^{-(t_k - \tau)(-\Delta)^{\gamma}} - e^{-(t_n - \tau)(-\Delta)^{\gamma}} \big] \rP[ u \cdot \nabla u](\xi,\tau) \,d\tau,\\
		J_3 &= -\int_{t_k}^{t_n} e^{-(t_n - \tau)(-\Delta)^{\gamma}} \rP [u \cdot \nabla u](\xi,\tau) \,d\tau,\\
		J_4 &=  \int_0^{t_k} \big[ e^{-(t_k - \tau)(-\Delta)^{\gamma}} - e^{-(t_n - \tau)(-\Delta)^{\gamma}} \big] \rP[ |u|^2 u](\xi,\tau) \,d\tau,\\
		J_5 &= -\int_{t_k}^{t_n} e^{-(t_n - \tau)(-\Delta)^{\gamma}} \rP [|u|^2 u](\xi,\tau) \,d\tau.
	\end{aligned}
$$
The next step is to estimate $J_i$ in $\cX^{0}_{a,\sigma}(\bR^3)$, for each $i = 1,\dots,5$. Indeed, for $i = 1$, we have
$$
	\begin{aligned}
		\Vert J_1 \Vert_{\cX^{0}_{a,\sigma}} %&= \int_{\bR^3} e^{a|\xi|^\frac{1}{\sigma}} | \cF[(e^{-t_n(-\Delta)^{\gamma}} - e^{-t_k(-\Delta)^{\gamma}}) u_0](\xi) | \,d\xi\\
		&= \int_{\bR^3} e^{a|\xi|^{\frac{1}{\sigma}}} |e^{-t_n|\xi|^{2\gamma}} - e^{-t_k |\xi|^{2\gamma}}||\widehat{u_0}(\xi)| \,d\xi\\
		&= \int_{\bR^3} e^{a|\xi|^{\frac{1}{\sigma}}} e^{-t_k|\xi|^{2\gamma}}(1 - e^{-(t_n - t_k) |\xi|^{2\gamma}})|\widehat{u_0}(\xi)| \,d\xi\\
		&\leq \int_{\bR^3} e^{a|\xi|^{\frac{1}{\sigma}}} (1 - e^{-(t_n - t_k) |\xi|^{2\gamma}})|\widehat{u_0}(\xi)| \,d\xi\\
		&\leq \int_{\bR^3} e^{a|\xi|^{\frac{1}{\sigma}}} (1 - e^{-(T^* - t_k) |\xi|^{2\gamma}})|\widehat{u_0}(\xi)| \,d\xi .
	\end{aligned}
$$
Thereby, $\lim_{n,k\rightarrow \infty} \|J_1 \|_{\cX^{0}_{a,\sigma}}=0$ by  dominated convergence theorem (since $u_0\in \cX^{0}_{a,\sigma}(\bR^3)$). 

In addition, for $i = 2$, one deduces
	\begin{align}\label{w40}
		\nonumber \Vert J_2 \Vert_{\cX^{0}_{a,\sigma}} &\leq \int_{\bR^3} \int_0^{t_k} e^{a|\xi|^{\frac{1}{\sigma}}} |e^{-(t_k - \tau)|\xi|^{2\gamma}} - e^{-(t_n - \tau)|\xi|^{2\gamma}}| |\cF [\rP(u \cdot \nabla u)](\xi,\tau)| \,d\tau d\xi\\
		\nonumber&\leq \int_{\bR^3} \int_0^{t_k} e^{a|\xi|^{\frac{1}{\sigma}}} |e^{-(t_k - \tau)|\xi|^{2\gamma}} - e^{-(t_n - \tau)|\xi|^{2\gamma}}| |\widehat{u \cdot \nabla u}(\xi,\tau)| \,d\tau d\xi\\
		\nonumber&\leq \int_{\bR^3} \int_0^{t_k} e^{a|\xi|^{\frac{1}{\sigma}}} (1 - e^{-(t_n - t_k)|\xi|^{2\gamma}}) |\widehat{u \cdot \nabla u}(\xi,\tau)| \,d\tau d\xi\\
		&\leq \int_{\bR^3} \int_0^{T^*} e^{a|\xi|^{\frac{1}{\sigma}}} (1 - e^{-(T^* - t_k)|\xi|^{2\gamma}}) |\widehat{u \cdot \nabla u}(\xi,\tau)| \,d\tau d\xi .
	\end{align}
On the other hand,  (\ref{lem:produto}), (\ref{lem:split}), (\ref{w41}), (\ref{w42}) and Hölder's inequality  imply that
\begin{align}\label{w14}
		\nonumber \int_{\bR^3} \int_0^{T^*} e^{a|\xi|^{\frac{1}{\sigma}}}  |\widehat{u \cdot \nabla u}(\xi,\tau)| \,d\tau d\xi
\nonumber &\leq  \int_0^{T^*} \|(u \otimes u)(\tau)\|_{\mathcal{X}_{a,\sigma}^1} \, d\tau \\
\nonumber &\leq C \int_0^{T^*} \|u (\tau)\|_{\mathcal{X}_{a,\sigma}^0}^{2-\frac{1}{2\gamma}} \|u (\tau)\|_{\mathcal{X}_{a,\sigma}^{2\gamma}}^{\frac{1}{2\gamma}} \, d\tau \\
&\leq CM_*^{2-\frac{1}{2\gamma}}[M^*]^{\frac{1}{2\gamma}}[T^*]^{1-\frac{1}{2\gamma}}<\infty,
 \end{align}
 for $\gamma>\frac{1}{2}$. As a consequence of (\ref{w40}) and (\ref{w14}), one infers $\lim_{n,k\rightarrow \infty}\Vert J_2 \Vert_{\cX^{0}_{a,\sigma}}=0$ (it is enough to apply dominated convergence theorem).
 
For $i =3$, by (\ref{lem:produto}),  (\ref{lem:split}), (\ref{w41}), (\ref{w42}) and Hölder's inequality,  we can write the following arguments:
$$
	\begin{aligned}
		\Vert J_3 \Vert_{\cX^{0}_{a,\sigma}} &\leq \int_{\bR^3} \int_{t_k}^{t_n} e^{a|\xi|^{\frac{1}{\sigma}}}e^{-(t_n - \tau)|\xi|^{2\gamma}}|\widehat{u \cdot \nabla u}(\xi,\tau)| \,d\tau d\xi\\
		&\leq \int_{\bR^3} \int_{t_k}^{t_n} e^{a|\xi|^{\frac{1}{\sigma}}}|\widehat{u \cdot \nabla u}(\xi,\tau)| \,d\tau d\xi\\
		&\leq \int_{t_k}^{T^*} \Vert u(\tau) \Vert_{\cX^{0}_{a,\sigma}} \Vert u(\tau) \Vert_{\cX^{1}_{a,\sigma}} \,d\tau\\
&\leq M_*^{2-\frac{1}{2\gamma}} \int_{t_k}^{T^*} \Vert u(\tau) \Vert_{\cX^{2\gamma}_{a,\sigma}}^{\frac{1}{2\gamma}} \,d\tau\\
&\leq M_*^{2-\frac{1}{2\gamma}} [M^{*}]^{\frac{1}{2\gamma}}(T^*-t_k)^{1-\frac{1}{2\gamma}},
	\end{aligned}
$$
for $\gamma>\frac{1}{2}$. Therefore, $\lim_{n,k\rightarrow \infty}\Vert J_3 \Vert_{\cX^{0}_{a,\sigma}}=0$.

Now, let us study the case $i=4$. Thus, we obtain
	\begin{align}\label{w44}
		\nonumber\Vert J_4 \Vert_{\cX^{0}_{a,\sigma}} &\leq \int_{\bR^3} \int_0^{t_k} e^{a|\xi|^{\frac{1}{\sigma}}} |e^{-(t_k - \tau)|\xi|^{2\gamma}} - e^{-(t_n - \tau)|\xi|^{2\gamma}}| |\cF [\rP(|u|^2 u)](\xi,\tau)| \,d\tau d\xi\\
		\nonumber&\leq \int_{\bR^3} \int_0^{t_k} e^{a|\xi|^{\frac{1}{\sigma}}} |e^{-(t_k - \tau)|\xi|^{2\gamma}} - e^{-(t_n - \tau)|\xi|^{2\gamma}}| |\widehat{|u|^2 u}(\xi,\tau)| \,d\tau d\xi\\
		\nonumber&\leq \int_{\bR^3} \int_0^{t_k} e^{a|\xi|^{\frac{1}{\sigma}}} (1 - e^{-(t_n - t_k)|\xi|^{2\gamma}}) |\widehat{|u|^2 u}(\xi,\tau)| \,d\tau d\xi\\
		&\leq \int_{\bR^3} \int_0^{T^*} e^{a|\xi|^{\frac{1}{\sigma}}} (1 - e^{-(T^* - t_k)|\xi|^{2\gamma}}) |\widehat{|u|^2 u}(\xi,\tau)| \,d\tau d\xi.
	\end{align}
On the other hand,  (\ref{lem:produto}) and (\ref{w41})  imply that
\begin{align}\label{w45}
		\nonumber \int_{\bR^3} \int_0^{T^*} e^{a|\xi|^{\frac{1}{\sigma}}}  |\widehat{|u|^2 u}(\xi,\tau)| \,d\tau d\xi
&=  \int_0^{T^*} \|[|u|^2 u](\tau)\|_{\mathcal{X}_{a,\sigma}^0} \, d\tau\\ 
&\leq C \int_0^{T^*} \|u (\tau)\|_{\mathcal{X}_{a,\sigma}^0}^{3} \, d\tau \leq CM_*^{3}T^*<\infty.
 \end{align}
 As a consequence of (\ref{w44}) and (\ref{w45}), one infers $\lim_{n,k\rightarrow \infty}\Vert J_4 \Vert_{\cX^{0}_{a,\sigma}}=0$ (it is enough to apply  dominated convergence theorem).

The last case $i = 5$ can be analysed as  follows: 
$$
	\begin{aligned}
		\Vert J_5 \Vert_{\cX^{0}_{a,\sigma}} &\leq \int_{\bR^3} \int_{t_k}^{t_n} e^{a|\xi|^{\frac{1}{\sigma}}}e^{-(t_n - \tau)|\xi|^{2\gamma}}|\widehat{|u|^2 u}(\xi,\tau)| \,d\tau d\xi\\
		&\leq \int_{\bR^3} \int_{t_k}^{t_n} e^{a|\xi|^{\frac{1}{\sigma}}}|\widehat{|u|^2 u}(\xi,\tau)| \,d\tau d\xi\\
		&\leq \int_{t_k}^{t_n} \Vert [|u|^2 u](\tau) \Vert_{\cX^{0}_{a,\sigma}} \,d\tau\\ 
		&\leq C\int_{t_k}^{T^*} \Vert u(\tau) \Vert_{\cX^{0}_{a,\sigma}}^3\,d\tau\\
& \leq C M_*^3 (T^*-t_k),
	\end{aligned}
$$
it is enough to apply the inequalities (\ref{lem:produto}) and (\ref{w41}). Hence,  $\lim_{n,k\rightarrow \infty}\Vert J_5 \Vert_{\cX^{0}_{a,\sigma}}=0$.

By the limits obtained above, we conclude that $(u(t_n))_{n\in\mathbb{N}}$ is a Cauchy sequence in $\cX^{0}_{a,\sigma}(\bR^3)$, which is a Banach space. Hence, there exists $u^* \in \cX^{0}_{a,\sigma}(\bR^3)$ such that
\begin{align}\label{w47}
	\lim_{n\rightarrow \infty}\Vert u(t_n) - u^* \Vert_{\cX^{0}_{a,\sigma}}=0.
\end{align}
%The preceding estimates are uniform for every pair $s,t$ sufficiently close to $T^*$; they
%do not depend on a particular  sequence.  Hence $u(t)$ is a Cauchy family in
%$\cX^0_{a,\sigma}(\mathbb{R}^3)$ as $t\nearrow T^*$ and has a unique limit $u^*$.  This directly proves
%that the limit is sequence-independent.  (The former interlacing argument was invalid because
%the interlaced sequence need not be increasing.)  Since divergence is preserved in
%$\cX^0_{a,\sigma}$, $\operatorname{div}u^*=0$.
We now need to show that the limit $u^*$ does not depend on the sequence $(t_n)_{n\in\mathbb{N}}$.
To this end, let $(s_n)_{n\in\mathbb{N}} \subseteq [0,T^*)$ such that $s_n \nearrow T^*$, as $n\rightarrow\infty$. By the process established above, we obtain $u^{**} \in \cX^{0}_{a,\sigma}(\bR^3)$ satisfying
\begin{align}\label{w48}
	\lim_{n\rightarrow\infty}\Vert u(s_n) - u^{**} \Vert_{\cX^{0}_{a,\sigma}} = 0.
\end{align}
Thereby, define the sequence $(r_n)_{n\in\mathbb{N}}\subseteq[0,T^*)$ by
$		r_{2n} := t_n $ and $		r_{2n - 1} := s_n$, for all $n\in \mathbb{N}$. Then, $r_n \nearrow T^*$, as $n\rightarrow\infty$, and, consequently, by the arguments described above, there is $u_* \in \cX^{0}_{a,\sigma}(\bR^3)$ such that
\begin{align}\label{w49}
	\lim_{n\rightarrow \infty}\Vert u(r_n) - u_* \Vert_{\cX^{0}_{a,\sigma}} = 0.
\end{align}
However, by (\ref{w47}), (\ref{w48}) and (\ref{w49}), notice that
\[
	u^* = \lim_{n\rightarrow\infty} u(t_n) = \lim_{n\rightarrow\infty} u(r_{2n}) = \lim_{n\rightarrow\infty} u(r_n) = u_*
\]
and also
\[
	u_* = \lim_{n\rightarrow\infty} u(r_n)  = \lim_{n\rightarrow\infty} u(r_{2n - 1})= \lim_{n\rightarrow\infty} u(s_n)  = u^{**},
\]
in $\cX^{0}_{a,\sigma}(\bR^3)$. Therefore, $u^* = u^{**}$. This means that $u^*$ does not rely on the sequence $(t_n)_{n\in\mathbb{N}}$.

Now, consider the following system of equations:
\begin{equation}\label{NSu*}	\tag{NS$u^*$}
\left\{
\begin{array}{l}
v_t
\;\!+\,
(-\Delta)^{\gamma}\,v
\,+\,
v \cdot \nabla v
\,+\, \alpha|v|^2v
\,+\,
\nabla \;\!q \:\!
\;=\;
0, \quad x\in \mathbb{R}^3, t>0;\\
%
%\mbox{} \vspace{-0.300cm} \\
%
%
%\mbox{} \vspace{-0.300cm} \\
%
%\theta_t
%\;\!\,+\,
%(-\Delta)^{\beta}\,\theta
%\,\,+\,
%u \cdot \nabla \theta
%\;=\;
%0, \quad x\in \mathbb{R}^3, t> 0,\\
%
%\mbox{} \vspace{-0.300cm} \\
%
\mbox{div}\:v  \;=\; 0, \quad x\in \mathbb{R}^3, t>0;\\
%
%\mbox{} \vspace{-0.300cm} \\
%
v(x,0) \,=\, u^*(x), \quad x\in \mathbb{R}^3.
\end{array}
\right.
\end{equation}
By Theorem \ref{thm:solucaolocal} iii), there exist $\bar T > 0$ and a unique local mild solution $v \in C_{\bar T}(\cX^{0}_{a,\sigma}(\bR^3)) \cap L^1_{\bar T}(\cX^{2\gamma}_{a,\sigma}(\bR^3))$ for the equations (\ref{NSu*}).
Then, by the limit (\ref{w47}), the function $w\in C_{\bar T+T^*}(\cX^{0}_{a,\sigma}(\bR^3))$ defined  by
$$
	w(t) =\left\{
         \begin{array}{ll}
           u(t), & \hbox{if }  t \in [0,T^*);\\
           v(t-T^*), & \hbox{if } t \in [T^*, T^* + \bar T],
         \end{array}
       \right.
$$
is a mild solution for the Navier--Stokes equations (\ref{NS}).  Indeed, passing to
the limit, as $t\nearrow T^*$, in the original mild formula gives
$$
u^*=e^{-T^*(-\Delta)^\gamma}u_0-
\int_0^{T^*}e^{-(T^*-\tau)(-\Delta)^\gamma}
\mathbb P\bigl(u\cdot\nabla u+\alpha|u|^2u\bigr)(\tau)\,d\tau.
$$
The semigroup identity combines this formula with the Duhamel formula for $v$ and yields the
mild formula for $w$ on $[0,T^*+\bar T]$.  This contradicts the maximality of $T^*$.
Therefore, we have proved that
$$
	\limsup_{t \nearrow T^*} \Vert u(t) \Vert_{\cX^{0}_{a,\sigma}} = \infty.
$$

\caixa

\bigskip
\noindent \underline{Proof of Theorem \ref{thm:blowup} ii)}: 
Consider that $T^*<\infty$. 
\\\\
Then, analogously to (\ref{w50}) and by (\ref{lem:split}), we can write
\begin{align} \label{w51}
	\nonumber\Vert u(s) \Vert_{\cX^{0}_{a,\sigma}} &+ \int_t^s \Vert u(\tau) \Vert_{\cX^{2\gamma}_{a,\sigma}} \,d\tau \leq \Vert u(t) \Vert_{\cX^{0}_{a,\sigma}} + C\int_t^s \Vert u(\tau) \Vert_{\cX^{0}_{a,\sigma}} \Vert u(\tau) \Vert_{\cX^{1}_{a,\sigma}} \,d\tau + \alpha C \int_{t}^s \Vert u(\tau) \Vert_{\cX^{0}_{a,\sigma}}^3 \,d\tau\\
&\quad\leq \Vert u(t) \Vert_{\cX^{0}_{a,\sigma}} + C\int_t^s  \Vert u(\tau) \Vert_{\cX^{0}_{a,\sigma}}^{2-\frac{1}{2\gamma}}\Vert u(\tau) \Vert_{\cX^{2\gamma}_{a,\sigma}}^{\frac{1}{2\gamma}} \,d\tau +\alpha C \int_{t}^s \Vert u(\tau) \Vert_{\cX^{0}_{a,\sigma}}^3 \,d\tau,
\end{align}
for $\gamma >\frac{1}{2}$, $\alpha\geq0$ and for all $0\leq t\leq s <T^*$. By Young's inequality, we obtain
	\begin{align*}
		\Vert u(s) \Vert_{\cX^{0}_{a,\sigma}} + \frac{1}{2} \int_t^s \Vert u(\tau) \Vert_{\cX^{2\gamma}_{a,\sigma}} \,d\tau 
		&\leq \Vert u(t) \Vert_{\cX^{0}_{a,\sigma}} + C_\gamma\int_t^s  \Vert u(\tau) \Vert_{\cX^{0}_{a,\sigma}}^{\frac{4\gamma-1}{2\gamma-1}} \,d\tau +\alpha C \int_{t}^s \Vert u(\tau) \Vert_{\cX^{0}_{a,\sigma}}^3 \,d\tau\\
		&\leq \Vert u(t) \Vert_{\cX^{0}_{a,\sigma}} + C_\gamma(\alpha + 1) \int_t^s [\Vert u(\tau) \Vert_{\cX^{0}_{a,\sigma}}^{\frac{4\gamma -1}{2\gamma -1}} + \Vert u(\tau) \Vert_{\cX^{0}_{a,\sigma}}^3] \,d\tau,
	\end{align*}
for $\gamma >\frac{1}{2}$, $\alpha\geq0$ and for all $0\leq t\leq s <T^*$. Therefore, by Grönwall's inequality, one deduces
\begin{equation} \label{eq:B2}
	\Vert u(s) \Vert_{\cX^{0}_{a,\sigma}}  \leq \Vert u(t) \Vert_{\cX^{0}_{a,\sigma}} \exp \Big\{ C_\gamma(\alpha + 1) \int_t^s [\Vert u(\tau) \Vert_{\cX^{0}_{a,\sigma}}^{\frac{2\gamma}{2\gamma -1}} + \Vert u(\tau) \Vert_{\cX^{0}_{a,\sigma}}^2] \,d\tau\Big\},
\end{equation}
for all $0\leq t\leq s <T^*$. By passing to the limit superior, as $s \nearrow T^*$, and applying Theorem \ref{thm:blowup} i), it follows that
\begin{equation*} 
	\int_t^{T^*}[ \Vert u(\tau) \Vert_{\cX^{0}_{a,\sigma}}^{\frac{2\gamma}{2\gamma - 1}} + \Vert u(\tau) \Vert_{\cX^{0}_{a,\sigma}}^2] \,d\tau = \infty,\quad\forall t\in[0,T^*).
\end{equation*}

\caixa

\bigskip
\noindent \underline{Proof of Theorem \ref{thm:blowup} iii)}: 
Consider that $T^*<\infty$. 
\\\\
Thereby, notice that, by using (\ref{eq:B2}), we can write the following inequality:
$$
	\Vert u(s) \Vert_{\cX^{0}_{a,\sigma}}^{\frac{2\gamma}{2\gamma - 1}} + \Vert u(s) \Vert_{\cX^{0}_{a,\sigma}}^2 \leq [\Vert u(t) \Vert_{\cX^{0}_{a,\sigma}}^{\frac{2\gamma}{2\gamma -1}} + \Vert u(t) \Vert_{\cX^{0}_{a,\sigma}}^2 ] \exp \Big\{ C_\gamma(\alpha + 1) \int_t^s [\Vert u(\tau) \Vert_{\cX^{0}_{a,\sigma}}^{\frac{2\gamma}{2\gamma -1}} + \Vert u(\tau) \Vert_{\cX^{0}_{a,\sigma}}^2 ] \,d\tau\Big\},
$$
for $\gamma > \frac{1}{2}$ and for all $0\leq t\leq s <T^*$. This inequality above is equivalent to
$$
	\frac{d}{ds} \Big[ -[C_\gamma(\alpha + 1)]^{-1} \exp \Big\{ -C_\gamma(\alpha + 1) \int_t^s[ \Vert u(\tau) \Vert_{\cX^{0}_{a,\sigma}}^{\frac{2\gamma}{2\gamma - 1}} +  \Vert u(\tau) \Vert_{\cX^{0}_{a,\sigma}}^2] \,d\tau \Big\}\Big] \leq \Vert u(t) \Vert_{\cX^{0}_{a,\sigma}}^{\frac{2\gamma}{2\gamma - 1}} + \Vert u(t) \Vert_{\cX^{0}_{a,\sigma}}^2,
$$
for all $0\leq t\leq s <T^*$. Hence, by integrating the inequality above over $[t,t_0]$ (with $0 \leq t \leq t_0 < T^*$), we reach
$$
	\begin{aligned}
		&-[C_\gamma (\alpha + 1)]^{-1} \exp \Big\{ -C_\gamma (\alpha + 1) \int_t^{t_0} [\Vert u(\tau) \Vert_{\cX^{0}_{a,\sigma}}^{\frac{2\gamma}{2\gamma - 1}} +  \Vert u(\tau) \Vert_{\cX^{0}_{a,\sigma}}^2 ]\,d\tau\Big\}
		+ [C_\gamma (\alpha + 1)]^{-1}\\
& \leq [\Vert u(t) \Vert_{\cX^{0}_{a,\sigma}}^{\frac{2\gamma}{2\gamma - 1}} +  \Vert u(t) \Vert_{\cX^{0}_{a,\sigma}}^2](t_0 - t),
	\end{aligned}
$$
for all $0\leq t\leq t_0 <T^*$. Therefore, by passing the limit, as $t_0 \nearrow T^*$, and applying Theorem \ref{thm:blowup} ii),  it follows that
$$
	\frac{[C_\gamma(\alpha + 1)]^{-1}}{T^* - t} \leq \Vert u(t) \Vert_{\cX^{0}_{a,\sigma}}^{\frac{2\gamma}{2\gamma - 1}} + \Vert u(t) \Vert_{\cX^{0}_{a,\sigma}}^2,\quad\forall t\in[0,T^*).
$$

\caixa

\noindent\textbf{Proof of Corollary \ref{corollaryB1}:} This result follows from Theorems \ref{thm:solucaolocal} iii) and \ref{thm:blowup} and the arguments established   by P. L. Guidolin, W. G. Melo and T. S. R. Santos \cite{patricia} for the Boussinesq equations.
The loss-of-radius step needed here is as follows.  Set
$U_r(t)=\Vert u(t)\Vert_{\mathcal X^0_{r,\sigma}}$,
$D_a(t)=\Vert u(t)\Vert_{\mathcal X^{2\gamma}_{a,\sigma}}$,
$\lambda=(2\gamma)^{-1}$ and $q=(1-\lambda)^{-1}=2\gamma/(2\gamma-1)$.
The product estimate (\ref{lem:produto}), whose constant is uniform in the radius, gives
$$
\Vert u\otimes u\Vert_{\mathcal X^1_{a,\sigma}}
\le C U_{a/\sigma}\Vert u\Vert_{\mathcal X^1_{a,\sigma}},
$$
and interpolation, followed by Young's inequality, yields
$$
C U_{a/\sigma}U_a^{1-\lambda}D_a^\lambda
\le\frac12D_a+C_\gamma U_{a/\sigma}^{q}U_a.
$$
Applying the same product estimate (\ref{lem:produto}) twice  also gives
$\Vert |u|^2u\Vert_{\mathcal X^0_{a,\sigma}}\le C U_{a/\sigma}^2U_a$.
Therefore, one concludes
$$
U_a(s)\le U_a(t)+C_\gamma(1+\alpha)\int_t^s
\bigl(U_{a/\sigma}^q+U_{a/\sigma}^2\bigr)U_a\,d\tau,
$$
and Gronwall's lemma gives the following inequality without any implication from a stronger
norm to a weaker one being assumed.
\begin{align}\label{n1}
\|u(s)\|_{\mathcal{X}_{a,\sigma}^{0}}&\leq 
 \|u(t)\|_{\mathcal{X}_{a,\sigma}^{0}}\exp \{C_{\gamma}(\alpha+1)\int_t^s [\|u(\tau)\|_{\mathcal{X}_{\frac{a}{\sigma},\sigma}^{0}}^{\frac{2\gamma}{2\gamma-1}}
+\|u(\tau)\|_{\mathcal{X}_{\frac{a}{\sigma},\sigma}^{0}}^{2}]\,d\tau \}, \quad\forall 0\leq t\leq s< T^*,
\end{align}
where $a\geq0$, $\sigma\geq1$, $\alpha\geq0$ and $\gamma>\frac{1}{2}$. Take the limit superior in (\ref{n1}), as $s\nearrow T^*$, in order to apply Theorem \ref{thm:blowup}  i) to conclude
\begin{align}\label{n2}
\int_t^{T^*} [\|u(\tau)\|_{\mathcal{X}_{\frac{a}{\sigma},\sigma}^{0}}^{\frac{2\gamma}{2\gamma-1}}
+\|u(\tau)\|_{\mathcal{X}_{\frac{a}{\sigma},\sigma}^{0}}^{2}]\,d\tau=\infty, \quad \forall t\in[0,T^*),
\end{align}
 Consequently, the proof of  Corollary  \ref{corollaryB1} i), with $n=1$, is complete.

By adapting (\ref{n1}), we can write the following inequality:
\begin{align*}
\|u(s)\|_{\mathcal{X}_{\frac{a}{\sigma},\sigma}^{0}}&\leq 
 \|u(t)\|_{\mathcal{X}_{\frac{a}{\sigma},\sigma}^{0}}\exp \{C_{\gamma}(\alpha+1)\int_t^s [\|u(\tau)\|_{\mathcal{X}_{\frac{a}{\sigma},\sigma}^{0}}^{\frac{2\gamma}{2\gamma-1}}
+\|u(\tau)\|_{\mathcal{X}_{\frac{a}{\sigma},\sigma}^{0}}^{2}]\,d\tau \},\quad\forall 0\leq t\leq s< T^*,
\end{align*}
where $a\geq0$, $\sigma\geq1$, $\alpha\geq0$ and $\gamma>\frac{1}{2}$.  As an immediate result, we have
\begin{align*}
\|u(s)\|_{\mathcal{X}_{\frac{a}{\sigma},\sigma}^{0}}^{\frac{2\gamma}{2\gamma-1}}+\|u(s)\|_{\mathcal{X}_{\frac{a}{\sigma},\sigma}^{0}}^{2}&\leq [\|u(t)\|_{\mathcal{X}_{\frac{a}{\sigma},\sigma}^{0}}^{\frac{2\gamma}{2\gamma-1}}+\|u(t)\|_{\mathcal{X}_{\frac{a}{\sigma},\sigma}^{0}}^{2}]\\
&\quad\times \exp\Big\{C_{\gamma}(\alpha+1)\int_{t}^{s}[\|u(\tau)\|_{\mathcal{X}_{\frac{a}{\sigma},\sigma}^{0}}^{\frac{2\gamma}{2\gamma-1}}
+\|u(\tau)\|_{\mathcal{X}_{\frac{a}{\sigma},\sigma}^{0}}^{2}]d\tau\Big\},
\end{align*}
for $\gamma >\frac{1}{2}$ and for all $0\leq t\leq s<T^*$.
After raising the preceding estimate to the powers
$q=\frac{2\gamma}{2\gamma-1}$ and $2$, replace the exponent constant by the fixed value
$K_\gamma=\max\{q,2\}C_\gamma$; below this enlarged constant is again denoted by
$C_\gamma$. Equivalently, it is true that
\begin{align*}
&\frac{d}{ds}\Big\{-[C_{\gamma}(\alpha+1)]^{-1}\exp\Big\{-C_{\gamma}(\alpha+1)\int_t^s[\|u(\tau)\|_{\mathcal{X}_{\frac{a}{\sigma},\sigma}^{0}}^{\frac{2\gamma}{2\gamma-1}}
+\|u(\tau)\|_{\mathcal{X}_{\frac{a}{\sigma},\sigma}^{0}}^{2}]d\tau\Big\}\Big]\leq \|u(t)\|_{\mathcal{X}_{\frac{a}{\sigma},\sigma}^{0}}^{\frac{2\gamma}{2\gamma-1}}
+\|u(t)\|_{\mathcal{X}_{\frac{a}{\sigma},\sigma}^{0}}^{2},
\end{align*}
for all $0\leq t\leq s<T^*$. By integrating  over $[t,t_0]$ the inequality above ($0 \leq t \leq t_0 < T^*$), we infer
\begin{align*}
&-[C_{\gamma}(\alpha+1)]^{-1}\exp\Big\{-C_{\gamma}(\alpha+1)\int_t^{t_0}[\|u(\tau)\|_{\mathcal{X}_{\frac{a}{\sigma},\sigma}^{0}}^{\frac{2\gamma}{2\gamma-1}}
+\|u(\tau)\|_{\mathcal{X}_{\frac{a}{\sigma},\sigma}^{0}}^{2}]d\tau\Big\}+[C_{\gamma}(\alpha+1)]^{-1}\\
&\leq 
[\|u(t)\|_{\mathcal{X}_{\frac{a}{\sigma},\sigma}^{0}}^{\frac{2\gamma}{2\gamma-1}}
+\|u(t)\|_{\mathcal{X}_{\frac{a}{\sigma},\sigma}^{0}}^{2}](t_0-t),
\end{align*}
for all $0\leq t\leq s<T^*$. By passing the limit in this last inequality, as $t_0\nearrow T^*$, and applying Corollary \ref{corollaryB1} i), with $n=1$, we deduce
\begin{align}\label{n3}
\|u(t)\|_{\mathcal{X}_{\frac{a}{\sigma},\sigma}^{0}}^{\frac{2\gamma}{2\gamma-1}}
+\|u(t)\|_{\mathcal{X}_{\frac{a}{\sigma},\sigma}^{0}}^{2}\geq \frac{[C_{\gamma}(\alpha+1)]^{-1}}{T^*-t}, \quad \forall t\in[0,T^*),
\end{align}
 Hence, the proof of Corollary  \ref{corollaryB1} ii), with $n=3$, is given.

From now on, for each radius $\omega\ge0$, let $T_\omega^*$ denote the maximal lifespan in the complete
time--frequency space associated with $\mathcal X^0_{\omega,\sigma}(\mathbb{R}^3)$.  If $0\le b<a$, the
embedding $\mathcal X^0_{a,\sigma}(\mathbb{R}^3)\subset\mathcal X^0_{b,\sigma}(\mathbb{R}^3)$ and local uniqueness imply
that the two maximal solutions agree on their common interval and that $T_b^*\ge T_a^*$.
Whenever the lower bound proved below gives
$\limsup_{t\nearrow T_a^*}U_b(t)=\infty$, the strict inequality $T_b^*>T_a^*$ would contradict
continuity of the radius-$b$ solution at $T_a^*$.  Hence $T_b^*\le T_a^*$ and the two maximal
times are equal.  This compatibility argument, rather than inclusion alone, justifies the
iteration.  The cases $a=0$ and $\sigma=1$ are immediate because the relevant radii coincide.

It is easy to check that $\displaystyle\|\cdot\|_{\mathcal{X}_{\frac{a}{(\sqrt{\sigma})^n},\sigma}^0}\leq \|\cdot\|_{\mathcal{X}_{a,\sigma}^0}$, for all $n\in \mathbb{N}$ (since $a\geq0$ and $\sigma\geq1$). Therefore, one concludes that $u\in C([0,T_{a}^*),$ $\mathcal{X}_{\frac{a}{(\sqrt{\sigma})^n},\sigma}^0(\mathbb{R}^3))$ (by hypothesis, $u\in C([0,T_{a}^*),$ $\mathcal{X}_{a,\sigma}^0(\mathbb{R}^3))$) and this leads us to obtain
\begin{align}\label{wilber10}
T_{\frac{a}{(\sqrt{\sigma})^n}}^*\geq T_a^*,\quad\forall n\in\mathbb{N}.
\end{align}
In addition,  (\ref{n3}) implies that
\begin{align*}
\nonumber \frac{[C_\gamma(\alpha+1)]^{-1}}{T^{*}_a-t}&\leq  \| u(t)\|_{\mathcal{X}^0_{\frac{a}{\sigma},\sigma}}^{\frac{2\gamma}{2\gamma-1}}+\| u(t)\|_{\mathcal{X}^0_{\frac{a}{\sigma},\sigma}}^{2}
\leq \|u(t)\|_{\mathcal{X}_{\frac{a}{\sqrt{\sigma}},\sigma}^0}^{\frac{2\gamma}{2\gamma-1}}+\|u(t)\|_{\mathcal{X}_{\frac{a}{\sqrt{\sigma}},\sigma}^0}^{2},
\end{align*}
for all $t\in [0,T_a^*)$ (for $a\geq0$ and $\sigma\geq1$). As a consequence, one can write
\begin{align}\label{i)n=1}
 \|u(t)\|_{\mathcal{X}_{\frac{a}{\sqrt{\sigma}},\sigma}^0}^{\frac{2\gamma}{2\gamma-1}}+\|u(t)\|_{\mathcal{X}_{\frac{a}{\sqrt{\sigma}},\sigma}^0}^{2}\geq \frac{[C_{\gamma}(\alpha+1)]^{-1}}{T^{*}_a-t},\quad \forall t\in [0,T_a^*).
\end{align}
This means that Corollary \ref{corollaryB1} ii), with $n=2$ (the case $n=1$ is given in Theorem \ref{thm:blowup} iii)), has been proved.

By the use of (\ref{i)n=1}), we have
\begin{align}\label{esqueci2}
\limsup_{t\nearrow T_a^*}\|u(t)\|_{\mathcal{X}_{\frac{a}{\sqrt{\sigma}},\sigma}^0}=\infty,
\end{align}
where $a\geq0$, $\sigma\geq1$, $\alpha\geq0$ and $\gamma>\frac{1}{2}$. This limit is exactly  Corollary \ref{corollaryB1} iii), with $n=2$ (the case $n=1$ has been established in Theorem \ref{thm:blowup} i)) and proves that
\begin{align}\label{wilber11a}
T_a^*\geq T_{\frac{a}{\sqrt{\sigma}}}^*.
\end{align}
 Observing (\ref{wilber10}) and (\ref{wilber11a}), one concludes
\begin{align}\label{wilber11}
T_a^*= T_{\frac{a}{\sqrt{\sigma}}}^*.
\end{align}

On the other hand, by replacing $a$ by $\frac{a}{\sqrt{\sigma}}$ in the proof of Corollary \ref{corollaryB1} i), with $n=1$, and applying (\ref{esqueci2}), one obtains the verification of Corollary \ref{corollaryB1} i), with $n=2$, similarly to (\ref{n2}).

 As a consequence, by using this Corollary \ref{corollaryB1} i), with $n=2$,  we can prove Corollary \ref{corollaryB1} ii), with $n=4$, analogously to (\ref{n3}).

%By choosing $\frac{a}{\sqrt{\sigma}}$ instead of $a$ in the proof of (\ref{i)n=1}) (apply Corollary \ref{corollaryB1} ii), with $n=4$), one deduces
%\begin{align*}
% \|u(t)\|_{\mathcal{X}_{\frac{a}{\sigma},\sigma}^0}^{\frac{2\gamma}{2\gamma-1}}
%+\|u(t)\|_{\mathcal{X}_{\frac{a}{\sigma},\sigma}^0}^{2}\geq \frac{[C_{\gamma}(\alpha+1)]^{-1}}{T^{*}_{\frac{a}{\sqrt{\sigma}}}-t},\quad \forall t\in [0,T_{\frac{a}{\sqrt{\sigma}}}^*).
%\end{align*}
% Consequently, (\ref{wilber11}) implies that
%\begin{align}\label{n4} 
% \|u(t)\|_{\mathcal{X}_{\frac{a}{\sigma},\sigma}^0}^{\frac{2\gamma}{2\gamma-1}}
%+\|u(t)\|_{\mathcal{X}_{\frac{a}{\sigma},\sigma}^0}^{2}\geq \frac{[C_{\gamma}(\alpha+1)]^{-1}}{T^{*}_{a}-t},\quad \forall t\in [0,T_{a}^*).
%\end{align}
%Then,  Corollary  \ref{corollaryB1} ii), with $n=2$, is established.

On the other hand, by passing  the limit superior in  (\ref{n3}), as   $t\nearrow T^*_a$, one infers
$$\displaystyle \limsup_{t\nearrow T_a^*}\|u(t)\|_{\mathcal{X}_{\frac{a}{\sigma},\sigma}^0(\mathbb{R}^3)}=\infty,$$
where $a\geq0$, $\sigma\geq1$, $\alpha\geq0$ and $\gamma>\frac{1}{2}$. The result above proves that Corollary \ref{corollaryB1} iii), with $n=3$, holds and, furthermore, $T^*_a \geq T^*_{\frac{a}{\sigma}}$. By applying  (\ref{wilber10}), we conclude that
$T^*_a = T^*_{\frac{a}{\sigma}}$.

If we follow the process above,  we prove inductively that $T^*_a = T^*_{\frac{a}{(\sqrt{\sigma})^n}}$, for all $n\in \mathbb{N}$, and also that Corollary \ref{corollaryB1} i)--iii) hold. The product constants used in this induction are uniform in the radius and hence in $n$.

 Now, we are ready to show Corollary \ref{corollaryB1} iv). In fact,  by using Corollary \ref{corollaryB1} ii), we obtain
\begin{align}\label{wilber2}
\frac{[C_\gamma(\alpha+1)]^{-1}}{T^{*}-t}&\leq \left(\int_{\mathbb{R}^3}e^{\frac{a}{(\sqrt{\sigma})^{n-1}}|\xi|^{\frac{1}{\sigma}}}|\hat{u}(t)|\;d\xi\right)^{\frac{2\gamma}{2\gamma-1}}
+\left(\int_{\mathbb{R}^3}e^{\frac{a}{(\sqrt{\sigma})^{n-1}}|\xi|^{\frac{1}{\sigma}}}|\hat{u}(t)|\;d\xi\right)^{2},
\end{align}
where $a\geq0$, $\sigma\geq1$, $\alpha\geq0$, $\gamma>\frac{1}{2}$ and for all $t\in[0,T^*),n\in\mathbb{N}$. In addition,   we know that
$$e^{\frac{a}{(\sqrt{\sigma})^{n-1}}|\xi|^{\frac{1}{\sigma}}}\leq e^{a|\xi|^{\frac{1}{\sigma}}},\quad\forall n\in\mathbb{N},$$
since that $a\geq0$ and $\sigma\geq1$. Therefore, by taking into account the limit in (\ref{wilber2}), as $n\rightarrow \infty$,  it follows, from dominated convergence Theorem, that
\begin{align*}
\frac{[C_\gamma(\alpha+1)]^{-1}}{T^{*}-t}&\leq\|u(t)\|_{\mathcal{X}^0}^{\frac{2\gamma}{2\gamma-1}}+ \|u(t)\|_{\mathcal{X}^0}^{2},\quad\forall t\in[0,T^*),
\end{align*}
since  $\sigma>1$. Hence, we complete the proof of Corollary \ref{corollaryB1} iv).
%(Notice that the finiteness of the norm above is assured by Lemma \ref{lemanovo2}).

\caixa

\bigskip
\noindent\textbf{Proof of Corollary \ref{corollaryB2}:}
This result is  a consequence of Corollary \ref{corollaryB1} iv) and an adaptation of the arguments presented   by P. L. Guidolin, W. G. Melo and T. S. R. Santos \cite{patricia} for the Boussinesq equations. Thus, by using Corollary \ref{corollaryB1} iv) (with $\gamma =1$), we can write the following inequality:
\begin{align}\label{estimativaalphabeta}
H(t):=\frac{[C(\alpha+1)]^{-\frac{1}{2}}}{(T^*-t)^{\frac{1}{2}}}\leq \|u(t)\|_{\mathcal{X}^0}, \quad \forall t\in[0,T^*),
\end{align}
where $a\geq0$, $\sigma>1$ and $\alpha\geq0$. The interpolation constant required below can be chosen independently of $\delta>0$: splitting the Fourier integral at a radius $R$ gives
$$
\Vert f\Vert_{\mathcal X^0}
\le C_LR^{3/2}\Vert f\Vert_{L^2}+R^{-\delta}\Vert f\Vert_{\mathcal X^\delta}.
$$
Choosing
$R=(\Vert f\Vert_{\mathcal X^\delta}/\Vert f\Vert_{L^2})^{2/(2\delta+3)}$
yields
\begin{equation}\label{wilber3}
\Vert f\Vert_{\mathcal X^0}
\le C_0\Vert f\Vert_{L^2}^{\frac{2\delta}{2\delta+3}}
\Vert f\Vert_{\mathcal X^\delta}^{\frac{3}{2\delta+3}},
\end{equation}
where $C_0=C_L+1$ is independent of $\delta\geq0$ (for more details, see (\ref{lemanovo2})).

The $L^2$ estimate for the mild solution is justified by smooth divergence-free Fourier
truncations.  Their exact energy identities, weak compactness, and lower semicontinuity give
$$
\frac12\Vert u(t)\Vert_{L^2}^2+
\int_0^t\Vert(-\Delta)^{1/2}u(\tau)\Vert_{L^2}^2\,d\tau+
\alpha\int_0^t\Vert u(\tau)\Vert_{L^4}^4\,d\tau
\le\frac12\Vert u_0\Vert_{L^2}^2.
$$
In particular,
\begin{equation}\label{normal2}
\Vert u(t)\Vert_{L^2}\le\Vert u_0\Vert_{L^2},\quad \forall t\geq0.
\end{equation}
Taking $\delta=\frac{k}{\sigma}$ in (\ref{wilber3}), for $k\in \mathbb{N}$, and using
$\Vert u(t)\Vert_{\mathcal X^0}\ge H(t)$ (see (\ref{estimativaalphabeta})) gives the uniform lower bound
\begin{equation}\label{W1}
\Vert u(t)\Vert_{\mathcal X^{\frac{k}{\sigma}}}
\ge \frac{H(t)}{C_0}
\left(\frac{H(t)}{C_0\Vert u_0\Vert_{L^2}}\right)^{\frac{2k}{3\sigma}}.
\end{equation}
The term $k=0$ is supplied separately by (\ref{estimativaalphabeta}).
As a consequence, by adding the $k=0$ estimate and multiplying (\ref{W1}) by $\frac{a^k}{k!}$ for $k\ge1$, one has
\begin{align*}
\frac{H(t)}{C_0}
\sum_{k\ge0}\frac{a^k}{k!}
\left(\frac{H(t)}{C_0\Vert u_0\Vert_{L^2}}\right)^{\frac{2k}{3\sigma}}
\leq \sum_{k\ge0}\frac{a^k}{k!}\|u(t)\|_{\mathcal{X}^{\frac{k}{\sigma}}},
\end{align*}
or equivalently,
\begin{align*}
\frac{H(t)}{C_0}
\exp\left\{a\left(\frac{H(t)}{C_0\Vert u_0\Vert_{L^2}}\right)^{\frac{2}{3\sigma}}\right\}
\leq \|u(t)\|_{\mathcal{X}_{a,\sigma}^{0}},\quad\forall t\in [0,T^*).
\end{align*}
This proves Corollary \ref{corollaryB2}.
%On the other hand, define
%$$f(x)=\Big[e^x-\displaystyle\sum_{k=0}^{n_0-1}\frac{x^k}{k!}\Big][x^{-n_0}e^{-\frac{x}{2}}],\quad \forall x>0.$$
%It is easy to check that  there exists a positive constant $C_{s,\sigma}$ such that $f(x)\geq C_{s,\sigma}$, for all $x>0$. Thus, in particular, by taking %$x=a[C_{\alpha,T^*}[e^{C_{\alpha}(T^*-t)}-1]^{-\frac{2\alpha-1}{2\alpha}}\|(u_0,\theta_0)\|_{L^2}^{-1}]^{\frac{2}{3\sigma}}$, we deduce
%\begin{align*}
%\|(u,\theta)(t)\|_{\mathcal{X}_{a,\sigma}^{s}}&\geq %a^{n_0}C_{s,\sigma,\alpha,T^*}[e^{C_{\alpha}(T^*-t)}-1]^{-\frac{2\alpha-1}{2\alpha}(\frac{2s}{3}+\frac{2n_0}{\sigma}+1)}\|(u_0,\theta_0)\|_{L^2}^{-\frac{2}{3}(s+\frac{n_0}{\sigma})}\\
%&\quad\times \exp\Big\{aC_{\sigma,\alpha,T^*}
%[e^{C_{\alpha}(T^*-t)}-1]^{-\frac{2\alpha-1}{3\alpha\sigma}}\|(u_0,\theta_0)\|_{L^2}^{-\frac{2}{3\sigma}}\Big\},\quad\forall t\in[0,T^*).
%\end{align*}

\caixa

%\noindent\textbf{Proof of Corollary \ref{corollaryB1} vi):} By applying  Corollary \ref{corollaryB1} iii), we obtain
%\begin{align*}
%C_{a,\sigma,s,\alpha,\beta}[e^{C_{\alpha,\beta}(T^*-t)}-1]^{-1}&\leq %\left(\int_{\mathbb{R}^3}|\xi|^se^{\frac{a}{(\sqrt{\sigma})^{n}}|\xi|^{\frac{1}{\sigma}}}|(\hat{u},\hat{\theta})(t)|\;d\xi\right)^{\frac{2\alpha}{2\alpha-1}}
%+\left(\int_{\mathbb{R}^3}|\xi|^se^{\frac{a}{(\sqrt{\sigma})^{n}}|\xi|^{\frac{1}{\sigma}}}|(\hat{u},\hat{\theta})(t)|\;d\xi\right)^{\frac{2\beta}{2\beta-1}},
%\end{align*}
%for all $t\in[0,T^*)$ and $n\in\mathbb{N}$, provided that $a>0$, $\sigma> 1$, $s\in[-1,0],$ $\alpha>\frac{1}{2}$ and $\beta>\frac{1}{2}$. On the other hand, it is easy to check that
%$$|\xi|^se^{\frac{a}{(\sqrt{\sigma})^{n}}|\xi|^{\frac{1}{\sigma}}}|(\hat{u},\hat{\theta})(\xi,t)|\leq|\xi|^se^{a|\xi|^{\frac{1}{\sigma}}}|(\hat{u},\hat{\theta})(\xi,t)|\in %L^1(\mathbb{R}^3),\quad\forall \xi\in \mathbb{R}^3,t\in[0,T^*),n\in\mathbb{N},$$
%provided that $a\geq0$, $\sigma\geq1$ and $s\in \mathbb{R}$.
%Hence, by taking the limit, as $n\rightarrow \infty$,  we infer
%\begin{align*}
%C_{a,\sigma,s,\alpha,\beta}[e^{C_{\alpha,\beta}(T^*-t)}-1]^{-1}&\leq\|(u,\theta)(t)\|_{\mathcal{X}^s}^{\frac{2\alpha}{2\alpha-1}}+ %\|(u,\theta)(t)\|_{\mathcal{X}^s}^{\frac{2\beta}{2\beta-1}},\quad\forall t\in[0,T^*).
%\end{align*}
%\caixa

\bigskip
\noindent\textbf{Proof of Theorem \ref{thm:solucaoestabilidade}:}
%\section{Proof of Theorem \ref{thm:solucaoestabilidade}} \label{sec:demonstracao3}
Let us study the stability of the global solution obtained in Theorem \ref{thm:solucaolocal} i). First of all, notice that (\ref{eq:H1}) implies that
\begin{align*}
\Vert v_0 \Vert_{\cX^{0}_{a,\sigma}}&\leq\Vert u_0 \Vert_{\cX^{0}_{a,\sigma}}+\Vert u_0 -v_0\Vert_{\cX^{0}_{a,\sigma}}< \frac{C'}{2}+\frac{C'}{2}=C',
\end{align*}
since $\Vert u_0 \Vert_{\cX^{0}_{a,\sigma}}< \frac{C'}{2}$ by hypothesis, where $C'$ is given in  Theorem \ref{thm:solucaolocal} i) (see (\ref{dadoinicial})).
Thus, we can apply Theorem \ref{thm:solucaolocal} i)  (recall that, by assumption, $v_0 \in \cX^{0}_{a,\sigma}(\bR^3)$ is divergence free) to obtain a global solution
\begin{align}\label{w2}
	v \in C([0,\infty), \cX^{0}_{a,\sigma}(\bR^3)) \cap L^1([0,\infty), \cX^{1}_{a,\sigma}(\bR^3))
\end{align}
for the Navier-Stokes equations (\ref{NSv}). Thereby, we are able to prove the inequality (\ref{estabilidade}). To this end, denote $\bar u = u - v$ and $\bar p = p - q$ (where $q$ is the pressure associated with $v$) to obtain
\begin{align}\label{w3}
	\bar u\in C([0,\infty), \cX^{0}_{a,\sigma}(\bR^3)) \cap L^1([0,\infty), \cX^{1}_{a,\sigma}(\bR^3)),
\end{align}
by (\ref{w1}) and (\ref{w2}), and also
%$$
%	\begin{aligned}
%		\bar u_t &= 
%u_t - v_t\\
%		&= - u \cdot \nabla u + v \cdot \nabla v - \nabla p + \nabla q - (-\Delta)^{\frac{1}{2}} u + (-\Delta)^{\frac{1}{2}}v\\
%		&= - u \cdot \nabla u + v \cdot \nabla v \textcolor{red}{\:-\: v \cdot \nabla u + v \cdot \nabla u}  - \nabla p + \nabla q - (-\Delta)^{\frac{1}{2}} u + (-\Delta)^{\frac{1}{2}}v\\
%		&= - (u - v) \cdot \nabla u - v \cdot \nabla (u- v) - \nabla(p - q) - (-\Delta)^{\frac{1}{2}} (u - v)\\
%		&= -\bar u \cdot \nabla u - v \cdot \nabla \bar u - \nabla \bar p - (-\Delta)^{\frac{1}{2}} \bar u\\
%		&= -\bar u \cdot \nabla u - v \cdot \nabla \bar u \textcolor{red}{\:+\: u \cdot \nabla \bar u - u \cdot \nabla \bar u} - \nabla \bar p - (-\Delta)^{\frac{1}{2}} \bar u\\
%		&=- \bar u \cdot \nabla u - (v-u) \cdot \nabla \bar u - u \cdot \nabla \bar u - \nabla \bar p - (-\Delta)^\frac{1}{2} \bar u\\
%		&=- \bar u \cdot \nabla u + \bar u \cdot \nabla \bar u - u \cdot \nabla \bar u - \nabla \bar p - (-\Delta)^\frac{1}{2} \bar u.
%	\end{aligned}
%$$
%Thus, we may consider the following system of equations
\begin{equation} \label{eq:A11}
	\left\{  
	\begin{aligned}
		&\bar u_t + \bar u \cdot \nabla u - \bar u \cdot \nabla \bar u + u \cdot \nabla \bar u + \nabla \bar p + (-\Delta)^{\frac{1}{2}} \bar u = 0, \quad x\in \mathbb{R}^3, t>0;\\
		&\hbox{div}\, \bar u  = 0,  \quad x\in \mathbb{R}^3, t>0;\\
		&\bar u(x, 0) = u_0(x) - v_0(x) =: \bar u_0(x),  \quad x\in \mathbb{R}^3.
	\end{aligned}
	\right.
\end{equation}
Applying the following computation first to Fourier-truncated smooth approximants and then passing to the limit by Fatou's lemma and dominated convergence, one infers
%$$
%  \hat u \cdot \hat u_t + |\xi| \big( \hat u \cdot \hat u \big) + \big( \hat u \cdot \widehat{u \cdot \nabla u} \big) + (\hat u \cdot \widehat{\nabla p}) =0.
%$$
%But since $\hat u \cdot \widehat{\nabla p} = 0$, we may write the equation above as
%$$
%  \hat u \cdot \hat u_t + |\xi| |\hat u|^2 +  \hat u \cdot \widehat{u \cdot \nabla u} = 0.
%$$
%Using the fact that $\partial_t |\hat u(t)|^2 = 2\Re(\hat u \cdot \hat u_t)$, we have
%$$
%  \begin{aligned}
%    \frac{1}{2} \partial_t |\hat u|^2 &= \Re \big( \hat u \cdot \hat u_t \big)\\
%    &= - |\xi| |\hat u|^2 - \Re \big( \hat u \cdot \widehat{u \cdot \nabla u} \big) \leq - |\xi| |\hat u|^2 + \big| \hat u \cdot \widehat{u \cdot \nabla u} \big|.
%  \end{aligned}
%$$
%Finally, by the Cauchy-Schwartz innequality, we obtain
\begin{equation*}
  \frac{1}{2} \partial_t |\hat {\bar{u}}|^2 + |\xi| |\hat {\bar{u}}|^2 \leq |\hat {\bar{u}}| |\widehat{ \bar{u}\cdot \nabla u}|+|\hat {\bar{u}}| |\widehat{ \bar{u}\cdot \nabla \bar{u}}|+|\hat {\bar{u}}| |\widehat{ u\cdot \nabla \bar{u}}|.
\end{equation*}
Therefore, for any $\delta > 0$, it follows that
%$$
%  \partial_t \big( |\hat u(t)|^2 + \delta \big)^{1/2} = \frac{\partial_t |\hat u(t)|^2}{2 (|\hat u(t)| + \delta)^{1/2}}.
%$$
%Hence, we can divide both sides of (\ref{eq:A5}) by $(|\hat u|^2 + \delta)^{1/2}$ to obtain
\begin{align*}
  \partial_t \big[ |\hat {\bar{u}}|^2 + \delta \big]^{1/2} + \frac{|\xi| |\hat {\bar{u}}|^2}{[|\hat {\bar{u}}|^2 + \delta]^{1/2}} &\leq \frac{|\hat{ \bar{u}}| |\widehat{\bar{u} \cdot \nabla u}|}{[|\hat {\bar{u}}|^2 + \delta]^{1/2}}+\frac{|\hat{ \bar{u}}||\widehat{\bar{u} \cdot \nabla \bar{u}|}}{[|\hat {\bar{u}}|^2 + \delta]^{1/2}}+\frac{|\hat{ \bar{u}}| |\widehat{u \cdot \nabla \bar{u}|}}{[|\hat {\bar{u}}|^2 + \delta]^{1/2}}\\
  &\leq |\widehat{\bar{u} \cdot \nabla u}|+|\widehat{\bar{u} \cdot \nabla \bar{u}}|+ |\widehat{u \cdot \nabla \bar{u}}|.
\end{align*}
Then, by integrating  over $[0, t]$  and passing the limit, as $\delta \to 0$, one can write
%$$
%  \int_{T_1}^{T^*} \partial_s \big[ |\hat u(\xi,s)|^2 + \delta \big]^{1/2}\,ds +  \int_{T_1}^{T^*}\frac{|\xi| |\hat u(\xi,s)|^2}{(|\hat u(\xi,s)|^2 + \delta)^{1/2}}\,ds \leq  %\int_{T_1}^{T^*}\frac{|\hat u(\xi,s)| |\widehat{u \cdot \nabla u|}}{(|\hat u(\xi,s)|^2 + \delta)^{1/2}} \,ds.
%$$
%By the dominated convergence theorem and the fundamental theorem of calculus, we obtain, as $\delta \to 0$
$$
  |\hat {\bar{u}}(t)| +  \int_{0}^{t} |\xi| |\hat {\bar{u}}(\tau)| \,d\tau \leq |\hat{ \bar{u}}_0| +  \int_{0}^{t} [|\widehat{\bar{u} \cdot \nabla u}(\tau)| +|\widehat{\bar{u} \cdot \nabla \bar{u}}(\tau)|+|\widehat{u \cdot \nabla \bar{u}}(\tau)|] \,d\tau,\quad\forall t\geq0.
$$
Finally, by multiplying the inequality  above by $e^{a|\xi|^{\frac{1}{\sigma}}}$ and integrating the result obtained over $\bR^3$, we have
\begin{align*}
  \Vert \bar{u}(t) \Vert_{\cX^{0}_{a,\sigma}} +  \int_{0}^{t} \Vert \bar{u}(\tau) \Vert_{\cX^{1}_{a,\sigma}} \,d\tau &\leq \Vert \bar{u}_0 \Vert_{\cX^{0}_{a,\sigma}} +  \int_{0}^{t} \Vert (\bar{u} \cdot \nabla u) (\tau)\Vert_{\cX^{0}_{a,\sigma}} \,d\tau+\int_{0}^{t} \Vert (\bar{u} \cdot \nabla \bar{u}) (\tau)\Vert_{\cX^{0}_{a,\sigma}} \,d\tau\\
  &\quad+\int_{0}^{t} \Vert (u \cdot \nabla \bar{u}) (\tau)\Vert_{\cX^{0}_{a,\sigma}} \,d\tau\\
  &\leq \Vert \bar{u}_0 \Vert_{\cX^{0}_{a,\sigma}} +  \int_{0}^{t} \Vert (u \otimes \bar u) (\tau)\Vert_{\cX^{1}_{a,\sigma}} \,d\tau+  \int_{0}^{t} \Vert (\bar{u} \otimes \bar{u}) (\tau)\Vert_{\cX^{1}_{a,\sigma}} \,d\tau\\
  &\quad+  \int_{0}^{t} \Vert (\bar u \otimes u) (\tau)\Vert_{\cX^{1}_{a,\sigma}} \,d\tau,
\end{align*}
%Lastly, we nee to estimate the norm $\Vert u \cdot \nabla u(\tau) \Vert_{\cX^{0}_{a,\sigma}}$. For any $u \in \cX^{0}_{a,\sigma}(\bR^3)$ we have
%$$
%	\widehat{u \cdot \nabla u} = \cF \bigg[ \sum_{j=1}^3 u_j \partial_j u \bigg] = \sum_{j=1}^3 \cF \big[ u_j \partial_j u \big].
%$$
%But since $\nabla \cdot u = 0$, it follows that
%$$
%	\widehat{u \cdot \nabla u} = \sum_{j=1}^3 \cF \big[ \partial_j (u_j u) \big] = \sum_{j=1}^3 i \xi_j \widehat{u_j u} = i \xi \cdot \widehat{u \otimes u}.
%$$
%Hence,
%\begin{equation} \label{eq:A6}
%	|\widehat{u \cdot \nabla u}| = | \xi \cdot \widehat{u \otimes u}| \leq |\xi| |\widehat{u \otimes u}|.
%\end{equation}
%Then, by (\ref{eq:A6}) and 
for all $t\geq0.$ Thus, by applying  (\ref{lem:produto}), we conclude
\begin{align} \label{eq:A12}
		\nonumber	\Vert \bar u(t) \Vert_{\cX^{0}_{a,\sigma}} &+ \int_0^t \Vert \bar{u}(\tau) \Vert_{\cX^{1}_{a,\sigma}} \,d\tau \leq \Vert \bar u_0 \Vert_{\cX^{0}_{a,\sigma}}\\ 
		&+C \int_0^t[ \Vert \bar u(\tau) \Vert_{\cX^{0}_{a,\sigma}} \Vert u(\tau) \Vert_{\cX^{1}_{a,\sigma}} + \Vert u(\tau) \Vert_{\cX^{0}_{a,\sigma}}\Vert \bar u(\tau) \Vert_{\cX^{1}_{a,\sigma}} + \Vert \bar u(\tau) \Vert_{\cX^{0}_{a,\sigma}} \Vert \bar u(\tau) \Vert_{\cX^{1}_{a,\sigma}}] \,d\tau,
	\end{align}
%\revision{Here $K_{\rm st}$ is fixed once and for all.  Since a smaller small-data constant remains admissible in Theorem \ref{thm:solucaolocal} i), we take $K_{\rm st}C'\leq\frac14$ and set $\eta_*=\frac{1}{4K_{\rm st}}$.}
%for all $ t \geq0$. 
On the other hand, by   (\ref{eq:H1}), one can write the following inequality:
$$
	\Vert \bar u_0 \Vert_{\cX^{0}_{a,\sigma}} = \Vert u_0 - v_0 \Vert_{\cX^{0}_{a,\sigma}} < \frac{C'}{2}.
$$
Hence, by applying (\ref{w3}),  there exists $T_1 >0$ such that
\begin{align}\label{w4}
	\Vert \bar u(\tau) \Vert_{\cX^{0}_{a,\sigma}} < \frac{C'}{2}=\frac{1}{32C},\quad\forall \tau \in [0, T_1].
\end{align}
(See (\ref{dadoinicial})). Thereby, consider that
\begin{align}\label{supremo}
	\bar T = \sup \Big\{ t \in [0,\infty) : \sup_{\tau \in [0,t]} \{\Vert \bar u(\tau) \Vert_{\cX^{0}_{a,\sigma}}\} < \frac{1}{4C} \Big\}
\end{align}
 to conclude that $0<  T_1 \leq \bar T $ (see (\ref{w4})). Thus, by (\ref{dadoinicial}) and the fact that $\Vert u_0 \Vert_{\cX^{0}_{a,\sigma}}<\frac{C'}{2}$, (\ref{w28w}) implies that
\begin{align}\label{w6}
\Vert u (\tau) \Vert_{\cX^0_{a,\sigma}} \leq 2 \Vert u_0 \Vert_{\cX^{0}_{a,\sigma}}<C'=\frac{1}{16C},\quad\forall \tau\in [0, t],
\end{align}
where $t\in [0,\bar T)$. As a result, by applying  (\ref{supremo}) and (\ref{w6}) to  (\ref{eq:A12}), it follows that
	\begin{align}\label{w7}
	\nonumber	\Vert \bar u(t) \Vert_{\cX^{0}_{a,\sigma}} &+ \int_0^t \Vert \bar u(\tau) \Vert_{\cX^1_{a,\sigma}} \,d\tau \leq \Vert \bar u_0 \Vert_{\cX^{0}_{a,\sigma}}\\ 
		&+ C\int_0^t \Vert \bar u(\tau) \Vert_{\cX^{0}_{a,\sigma}} \Vert u(\tau) \Vert_{\cX^1_{a,\sigma}} \,d\tau+ \frac{1}{4} \int_0^t \Vert \bar u(\tau) \Vert_{\cX^1_{a,\sigma}} \,d\tau + \frac{1}{4} \int_0^t \Vert \bar u(\tau) \Vert_{\cX^1_{a,\sigma}} \,d\tau,
	\end{align}
for all $0 \leq t < \bar T$.
Consequently, one can write
$$
	\Vert \bar u(t) \Vert_{\cX^{0}_{a,\sigma}} + \frac{1}{2}\int_0^t \Vert \bar u(\tau) \Vert_{\cX^1_{a,\sigma}} \,d\tau \leq \Vert \bar u_0 \Vert_{\cX^{0}_{a,\sigma}} + C\int_0^t \Vert \bar u(\tau) \Vert_{\cX^{0}_{a,\sigma}} \Vert u(\tau) \Vert_{\cX^1_{a,\sigma}} \,d\tau,\quad\forall 0 \leq t < \bar T.
$$
%In order to use Grönwall's Lemma, we have to rewrite the equation above in the following manner 
%$$
%	\begin{aligned}
%		\Vert \bar u(\tau) \Vert_{\cX^{0}_{a,\sigma}} &+ \frac{1}{4}\int_0^t \Vert \bar u(\tau) \Vert_{\cX^1_{a,\sigma}} \,d\tau \leq \Vert u_0 \Vert_{\cX^{0}_{a,\sigma}} \int_0^t \left( %\Vert \bar u(\tau) \Vert_{\cX^{0}_{a,\sigma}} + \frac{1}{4}\int_0^t \Vert \bar u(\tau) \Vert_{\cX^1_{a,\sigma}} \,d\tau \right) \Vert u(\tau) \Vert_{\cX^1_{a,\sigma}} \,d\tau,
%	\end{aligned}
%$$
%for all $0 \leq t < \bar T$.
Hence, by using Grönwall's inequality, we obtain
\begin{equation} \label{eq:A13}
	\Vert \bar u(t) \Vert_{\cX^{0}_{a,\sigma}} + \frac{1}{2} \int_0^t \Vert \bar u(\tau) \Vert_{\cX^{1}_{a,\sigma}}\,d\tau \leq \Vert \bar u_0 \Vert_{\cX^{0}_{a,\sigma}} \exp \left( C \int_0^t \Vert u(\tau) \Vert_{\cX^1_{a,\sigma}} \,d\tau \right) ,\quad \forall 0 \leq t < \bar T.
\end{equation}
We choose $C'':= C$, where $C>0$ is given in (\ref{eq:A13}) (see (\ref{eq:H1})).

Now, we shall prove that $\bar T = \infty$. In fact, suppose, by contradiction, that $\bar T < \infty$.
Since  $\bar u \in C([0,\infty), \cX^{0}_{a,\sigma}(\bR^3))$ (see (\ref{w3})), it follows, in particular, that $\bar u \in C_{\bar T}(\cX^{0}_{a,\sigma}(\bR^3))$ and, by (\ref{eq:H1}), (\ref{eq:A13}) and (\ref{dadoinicial}), we have
$$
	\Vert \bar u (\bar T) \Vert_{\cX^{0}_{a,\sigma}} = \lim_{t \nearrow \bar T} \Vert \bar u(t) \Vert_{\cX^{0}_{a,\sigma}} < \frac{C'}{2}=\frac{1}{32C}.
$$
By applying once more that $\bar u \in C([0,\infty),\cX^{0}_{a,\sigma}(\bR^3))$, one concludes that there exists $ T_2>\bar T$ such that
\begin{equation} \label{eq:A14}
	\Vert \bar u(\tau) \Vert_{\cX^{0}_{a,\sigma}} < \frac{1}{32C},\quad \forall \tau \in [\bar T,  T_2].
\end{equation}
Moreover, (\ref{eq:H1}), (\ref{eq:A13}) and (\ref{dadoinicial}) also imply
\begin{equation} \label{eq:B4}
	\Vert \bar u(\tau) \Vert_{\cX^{0}_{a,\sigma}} <\frac{C'}{2}=\frac{1}{32C},\quad \forall \tau \in [0,\bar T).
\end{equation}
Consequently, these last inequalities (\ref{eq:A14}) and (\ref{eq:B4}) lead us to infer
$$\sup_{\tau \in [0,T_2]} \{\Vert \bar u(\tau) \Vert_{\cX^{0}_{a,\sigma}}\} \leq \frac{1}{32C}<\frac{1}{4C}.$$ 
This proves that $ T_2 \leq \bar T$, which is a contradiction. Therefore,
$\bar T = \infty$ and, by  (\ref{eq:A13}) once again, we also have
%$$
%	\Vert \bar u(t) \Vert_{\cX^{0}_{a,\sigma}} + \frac{1}{2} \int_0^t \Vert \bar u(\tau) \Vert_{\cX^{1}_{a,\sigma}} \,d\tau < \frac{C'}{2},\quad \forall t \geq0.
%$$
%From the inequality above, one deduces that
%$$
%	\int_0^{T_*} \Vert \bar u(\tau) \Vert_{\cX^{1}_{a,\sigma}} \,d\tau \leq \frac{1}{2C} < \infty.
%$$
%Hence, we can write 
%\begin{align}\label{w9}
%	\int_{0}^{T_*} \Vert v(\tau) \Vert_{\cX^{1}_{a,\sigma}} \,d\tau \leq \int_{0}^{\infty} \Vert u(\tau) \Vert_{\cX^{1}_{a,\sigma}} \,d\tau + \int_{0}^{T_*} \Vert \bar u(\tau) %\Vert_{\cX^{1}_{a,\sigma}} \,d\tau < \infty.
%\end{align}
%Thereby, Theorem \ref{thm:solucaoglobal} i) (see (\ref{w8})) implies that $T_* = \infty$ and (\ref{eq:A13}) becomes
$$
	\Vert \bar u(t) \Vert_{\cX^{0}_{a,\sigma}} + \frac{1}{2} \int_0^t \Vert \bar u(\tau) \Vert_{\cX^{1}_{a,\sigma}}\,d\tau \leq \Vert \bar u_0 \Vert_{\cX^{0}_{a,\sigma}} \exp \left( C'' \int_0^t \Vert u(\tau) \Vert_{\cX^1_{a,\sigma}} \,d\tau \right),\quad \forall t\geq0.
$$ 
This establishes the proof of Theorem \ref{thm:solucaoestabilidade}.
%$$
%	\Vert u(t) - v(t) \Vert_{\cX^{0}_{a,\sigma}} + \frac{1}{2} \int_0^t \Vert u(\tau) - v(\tau) \Vert_{\cX^{1}_{a,\sigma}} \,d\tau \leq \Vert u_0 - v_0 \Vert_{\cX^{0}_{a,\sigma}} \exp \left( %C''\int_0^\infty \Vert u(\tau) \Vert_{\cX^{1}_{a,\sigma}}\,d\tau \right), \quad \forall t\geq0.
%$$

\caixa

\noindent \textbf{Declarations:}\\

\noindent \textbf{Ethics approval and consent to participate:}
%\subsection*{Ethics approval and consent to participate} 
Not applicable.\\

\noindent \textbf{Consent for publication:}
%\subsection*{Consent for publication} 
Not applicable.\\

\noindent \textbf{Availability of data and materials:}
%\subsection*{Availability of data and materials} 
Not applicable (this manuscript does not report data generation or analysis).\\

\noindent \textbf{Conflicts of interest/Competing interests:}
%\subsection*{Conflicts of interest/Competing interests} 
The authors have no conflicts of interest to declare that are relevant to the content of this article.\\

\noindent \textbf{Funding:}
%\subsection*{Funding} 
T.S.R. Santos is partially supported by São Paulo Research Foundation (FAPESP) grant 2024/15587-1.\\

\noindent \textbf{Authors' contributions:}
%\subsection*{Authors' contributions} 
W.M.,  B.D. and T.S.  wrote and reviewed the manuscript.\\

\noindent \textbf{Acknowledge:}
%\subsection*{Funding} 
B.S.D. was partially supported by a travel grant for a international traineeship by the University of Helsinki.\\

\end{document}